\documentclass[11pt, reqno]{amsart}

\usepackage{amssymb}
\usepackage{graphicx}
\usepackage{color}
\usepackage{enumitem}
\usepackage[all]{xy}

\newcommand{\C}{\mathbb{C}}
\newcommand{\R}{\mathbb{R}}
\newcommand{\Q}{\mathbb{Q}}
\newcommand{\Z}{\mathbb{Z}}
\newcommand{\F}{\mathbb{F}}
\newcommand{\X}{\mathcal{X}}
\newcommand{\OO}{\mathcal{O}}
\newcommand{\SL}{\mathit{SL}}
\newcommand{\PSL}{\mathit{PSL}}
\newcommand{\tr}{\operatorname{tr}}

\newcommand{\rank}{\operatorname{rank}}
\newcommand{\irr}{\mathrm{irr}}
\newcommand{\orb}{\mathrm{orb}}
\newcommand{\ab}{\mathrm{ab}}
\newcommand{\id}{\mathrm{id}}
\newcommand{\Int}{\operatorname{Int}}
\newcommand{\Fix}{\mathrm{Fix}}

\newcommand{\cox}{\mathrm{cox}}
\newcommand{\col}{\mathrm{col}}

\newcommand{\ang}[1]{\langle#1\rangle}

\newcommand{\aang}[1]{\langle\!\langle#1\rangle\!\rangle}
\newtheorem{theorem}{Theorem}[section]
\newtheorem{proposition}[theorem]{Proposition}
\newtheorem{corollary}[theorem]{Corollary}
\newtheorem{lemma}[theorem]{Lemma}
\newtheorem{claim}[theorem]{Claim}

\newtheorem{question}[theorem]{Question}
\theoremstyle{definition}
\newtheorem{definition}[theorem]{Definition}

\theoremstyle{remark}
\newtheorem{remark}[theorem]{Remark}
\begin{document}

\title{A preorder on the set of links with applications to symmetric unions}

\author{Michel Boileau}
\address{Aix Marseille Univ, CNRS, I2M, 3 Pl. Victor Hugo, 13003 Marseille, France}
\email{michel.boileau@univ-amu.fr}

\author{Teruaki Kitano}
\address{Department of Information Systems Science, Faculty of Science and Engineering, Soka University \\
Tangi-cho 1-236, Hachioji, Tokyo 192-8577 \\
Japan}
\email{kitano@soka.ac.jp}

\author{Yuta Nozaki}
\address{
Department of Mathematics, Faculty of Science, Hokkaido University \\
Kita-Ku, Sapporo 060-0810 \\
Japan\vspace{-0.6em}}
\address{
International Institute for Sustainability with Knotted Chiral Meta Matter (WPI-SKCM$^2$), Hiroshima University \\
1-3-1 Kagamiyama, Higashi-Hiroshima, Hiroshima 739-8531 \\
Japan}
\email{nozaki@math.sci.hokudai.ac.jp}
%\address{
%Faculty of Environment and Information Sciences, Yokohama National University \\
%79-7 Tokiwadai, Hodogaya-ku, Yokohama 240-8501 \\
%Japan\vspace{-0.6em}}
%\address{
%WPI-SKCM$^2$, Hiroshima University \\
%1-3-1 Kagamiyama, Higashi-Hiroshima, Hiroshima 739-8526 \\
%Japan}
%\email{nozakiy@hiroshima-u.ac.jp}

%57K10 Knot theory
%57K32 Hyperbolic 3-manifolds
%57R18 Topology and geometry of orbifolds
%57M12 Low-dimensional topology of special (e.g., branched) coverings
%57K30 General topology of 3-manifolds
%57R30 Foliations in differential topology; geometric theory
%57Q10 Simple homotopy type, Whitehead torsion, Reidemeister-Franz torsion, etc.


\maketitle

\begin{abstract}
For a link $L$ in the $3$-sphere, the $\pi$-orbifold group $G^\mathrm{orb}(L)$ is defined as a quotient of the link group $G(L)$ of $L$.
When there exists an epimorphism $G^\mathrm{orb}(L)\to G^\mathrm{orb}(L')$ fitting into a certain commutative diagram, we define a relation $L\succeq L'$ and explore the relationships between the two links.
Specifically, we prove that if $L\succeq L'$ and $L$ is a Montesinos link with $r$ rational tangles $(r\geq 3)$, then $L'$ is either a Montesinos link with at most $r+1$ rational tangles or a certain connected sum. 
We further show that if $L$ is a small link, then there are only finitely many links $L'$ satisfying $L\succeq L'$.
In contrast, if $L$ has determinant zero, then $L\succeq L'$ for every $2$-bridge link $L'$.
Our main applications concern symmetric unions of knots. 
In particular, we provide a criterion showing that a given knot does not admit a symmetric union presentation.
\end{abstract}

\setcounter{tocdepth}{1}
\tableofcontents

%%%%%%%%%%%%%%%%%%%%%%%%%%
\section{Introduction}
\label{sec:intro}

The goal of this paper is to study an order on the set of prime links (in the sense of \cite{Kaw96}, see Definition~\ref{def:prime} for the precise definition) with at least three bridges induced by epimorphisms between their $\pi$-orbifold groups. 
In Section~\ref{sec:symmetric_uinon}, we will give some applications to the study of symmetric unions of knots (see \cite{KiTe57}, \cite{Lam00}, \cite{BKN25}).

For a link $L$ in the $3$-sphere $S^3$, we write $E(L)$ for the exterior of $L$.
Let $G(L)$ denote the fundamental group of $E(L)$.
One can associate to a link $L$ the \emph{$\pi$-orbifold group} $G^\orb(L) = G(L)/N$, where $N$ is the subgroup of $G(L)$ normally generated by the squares of all meridians of $L$ (see \cite{BoZi89}). 
Here, a \emph{meridian} is an element of $G(L)$ which is represented by a curve that is freely homotopic to a meridional simple closed curve on $\partial E(L)$.
The group $G^\orb(L)$ is the orbifold fundamental group of the closed $3$-dimensional orbifold $\OO(L)$ with underlying space $S^3$ and singular locus $L$ with branching index $2$.
Let $\Sigma_2(L)$ denote the $2$-fold cover of $S^3$ branched along $L$.
Then the orbifold $\OO(L)$ is the quotient of the manifold $\Sigma_2(L)$ by the covering involution $\tau$.
Therefore, we have an exact sequence
\[
1 \to \pi_1(\Sigma_2(L)) \to G^\orb(L) \xrightarrow{\nu} \Z/2\Z \to 1,
\]
where the epimorphism $\nu$ sends the image of each meridian to the generator of $\Z/2\Z$. 
If $\Sigma_2(L)$ is aspherical, $\pi_1(\Sigma_2(L))$ is torsion-free and the only torsion elements in $G^{\orb}(L)$ come from the meridians, and hence they are of order $2$.

%%%%%%%%%%%%
\subsection{A preorder defined via orbifolds}
\label{subsec:order}

Let $L$ and $L'$ be links in $S^3$.

\begin{definition}
\label{def:domination}
$L$ \emph{$\pi$-dominates} $L'$, written $L \succeq L'$, if there is an epimorphism $\varphi\colon G^\orb(L) \twoheadrightarrow G^\orb(L')$ fitting into the commutative diagram
\[
\xymatrix{
G^\orb(L) \ar@{->>}[r]^{\nu} \ar@{->>}[d]_-{\varphi} & \Z/2\Z \ar[d]^-{\id} \\
G^\orb(L') \ar@{->>}[r]^{\nu'} & \Z/2\Z.}
\]
\end{definition}

See Section~\ref{subsec:pi-domination} for basic properties of the $\pi$-domination.
It is worth noting that a symmetric union construction gives plenty of examples of pairs $(K, K')$ of knots such that $K \succeq K'$ (see Section~\ref{sec:symmetric_uinon}).
By \cite{BoPo01} and \cite[Theorem~3.1]{BoZi89}, the $\pi$-orbifold group $G^\orb(L)$ determines the pair $(S^3, L)$ up to homeomorphism when $L$ is a prime link with at least three bridges. 
The $3$-manifold group $\pi_1(\Sigma_2(L))$ is residually finite by \cite{Hem87} and of index $2$ in $G^\orb(L)$.
It follows that the group $G^\orb(L)$ is a (finitely generated) residually finite group, and hence hopfian.
Then the relation $\succeq$ is a partial order on the set of prime links with at least three bridges, up to mirror image. 
The reflexivity and transitivity follow immediately from the definition.
The anti-symmetry is not obvious, but it follows from the facts that $\pi$-orbifold groups are hopfian and that prime links with at least three bridges are determined by their $\pi$-orbifold groups.

In our first theorem, we consider various classes of links.

\begin{definition}
A link $L \subset S^3$ is called
\begin{enumerate}[label=(\arabic*)]
\item a \emph{Seifert link} if its exterior $E(L)$ admits a Seifert fibration by circles (see \cite{BuMu70});
\item a \emph{Montesinos link} with $r$ rational tangles if $\Sigma_2(L)$ is Seifert fibered with base $S^2$ and with $r$ exceptional fibers such that the fiber-preserving covering involution reverses the orientation of the base and of the fibers (see \cite[Chapter~12]{BZH14}, \cite{Mon73}).
A Montesinos link $L$ with at least three rational tangles is said to be \emph{elliptic} if $\Sigma_2(L)$ is elliptic. 
\item a \emph{$\pi$-hyperbolic link} if $\Sigma_2(L)$ admits a hyperbolic structure.
\end{enumerate}
\end{definition}

\begin{remark}
A Montesinos link $L=L(\frac{\beta_1}{\alpha_1},\dots, \frac{\beta_r}{\alpha_r})$ is elliptic if and only if $r=3$ and the inequality $\sum_{i=1}^{3}\frac{1}{\alpha_i}>1$ holds (see Section~\ref{subsec:Montesinos} for the notation). 
\end{remark}

Here, the \emph{rank} $\rank(G)$ of a finitely generated group $G$ is defined to be the minimal number of elements needed to generate $G$.

\begin{theorem}\label{thm:order}
Let $L$ and $L'$ be two links such that $L \succeq L'$.
Then the following hold.
\begin{enumerate}[label=\textup{(\arabic*)}]
\item If $L$ is the unknot, then $L'$ is the unknot.
\item If $L$ is a $2$-bridge link, then $L'$ is a $2$-bridge link or the unknot.
\item If $L$ is a Montesinos link with $r$ rational tangles \textup{(}$r \geq 3$\textup{)}, then $L'$ is the unknot, or a $2$-bridge link, or a Montesinos link with $r'$ rational tangles \textup{(}$r'\leq r+1$\textup{)}, or a connected sum of $n_1$ $2$-bridge links and $n_2$ elliptic Montesinos links with $n_1 + 2n_2 \leq r-1$.
\item If $L$ is a Seifert link whose determinant is non-zero, then $L'$ is the unknot, or a $2$-bridge link, or an elliptic Montesinos link or a Seifert link, or a connected sum of $n_1$ $2$-bridge links and $n_2$ elliptic Montesinos links with $n_1 + 2n_2 \leq \rank \pi_1(\Sigma_2(L))$. 
\end{enumerate}
\end{theorem}

\begin{remark}\label{rem:rcomp}
In the case (3) of Montesinos links, if $L'$ has $r+1$ rational tangles, then $r$ is odd and $L' = L(\frac{\beta}{2\alpha + 1}, \frac{1}{2}, \dots, \frac{1}{2})$ by \cite[Theorem~1.1(i)]{BoZi84}. 
It follows that $L'$ is a link with $r$ components and thus $L$ has $r$ components (see Section~\ref{subsec:Montesinos}).
\end{remark}

\begin{corollary}\label{cor:bridge_number} 
If $K$ is a Montesinos knot and $K \succeq K'$, then $b(K) \geq b(K')$, where $b(K)$ denotes the bridge number of $K$.
\end{corollary}

In Section~\ref{sec:bridge}, we discuss the relationship between the bridge number and the preorder $\succeq$.

%%%%%%%%%
\subsection{Arborescent links}
$2$-bridge links, Montesinos links, and Seifert links belong to a much wider class of links called arborescent links. 
A link $L$ is an \emph{arborescent link} if $\Sigma_2(L)$ is a graph manifold, that is to say the geometric decomposition of $\Sigma_2(L)$ along essential 2-spheres and tori involves only Seifert fibered pieces.
See \cite{BoSi16}, \cite[Chapter~10.7]{Kaw96}, and \cite{Sak20}. 
For example, the Kinoshita-Terasaka knot is an arborescent knot. 
We also get some constraints on a link $\pi$-dominated by an arborescent link.

\begin{theorem}\label{thm:arb}
Let $L$ be an arborescent link with $\det L \neq 0$.
For a link $L'$ other than the unknot, if $L \succeq L'$, then each prime factor of the connected sum decomposition of the $2$-fold branched cover $\Sigma_2(L')$ has at least one Seifert fibered JSJ piece.
In particular, no factor of a connected sum or a split decomposition of $L'$ can be a $\pi$-hyperbolic link.
\end{theorem}

The proof relies crucially on Theorem~\ref{thm:graphtohyperbolic}, which is itself of independent interest.

\begin{question}\label{ques:arborescentdomination} 
Let $L$ be an arborescent link and let $L'$ be a prime link such that $L \succeq L'$ and $\det L \neq 0$.
Does it imply that $L'$ is an arborescent link?
\end{question}

%%%%%%%%
\subsection{Small links and finiteness}
\label{subsec:small_link}

A natural question to ask is whether a given link $L \subset S^3$ $\pi$-dominates only finitely many links in $S^3$. 
In Section~\ref{sec:small}, we show that the answer is no when $\det L = 0$ (see Proposition~\ref{prop:infinite}).

\begin{question}
Does a given link $L \subset S^3$ with non-zero determinant $\pi$-dominate only finitely many links in $S^3$? 
In particular, is it true for any knot $K \subset S^3$?
\end{question}

This question is widely open. 
We give a positive answer when $L$ is a small link, which means that its exterior $E(L)$ does not contain any compact, properly embedded, essential surface whose boundary is empty or a collection of meridian curves. 
In particular, a small link is prime or the unknot (see Section~\ref{sec:small} for more details).

\begin{theorem}
\label{thm:small}
Let $L$ be a small link.
\begin{enumerate}[label=\textup{(\arabic*)}]
\item If $L \succeq L'$, then $L'$ is a small link.
\item $L$ $\pi$-dominates only finitely many links in $S^3$.
\end{enumerate}
\end{theorem}

For instance, $2$-bridge knots, torus knots, and Montesinos knots with three rational tangles are small, hence they $\pi$-dominate only finitely many knots.

%%%
\subsection{Symmetric unions}\label{subsec:symmetricunion}
Finally, we use the $\pi$-orbifold group to study a classical construction in knot theory introduced by Kinoshita and Terasaka in \cite{KiTe57} and generalized by Lamm~\cite{Lam00}, which always produces ribbon knots. 
This construction is a generalization of the connected sum of a knot with its mirror image and called a \emph{symmetric union}. 
It gives a $\pi$-domination between the symmetric union and the partial knot (see Section~\ref{sec:symmetric_uinon}). 
Hence, our results (Theorems~\ref{thm:order}, \ref{thm:arb}, and \ref{thm:small}) provide strong constraints for a knot to be the partial knot of a symmetric union (see Corollaries~\ref{cor:partial_knot}, \ref{cor:arbpartial}, and \ref{cor:finiteness}). 
These tools can be applied to the open problem of whether every ribbon knot admits a symmetric union presentation.

The notion of $\pi$-minimal links developed in Section~\ref{sec:piminimal} leads to the following criterion for a knot not to be a symmetric union.

\begin{theorem}\label{thm:notsymmetricunion}
Let $K$ be a $\pi$-minimal knot. 
If $\rank H_1(\Sigma_2(K); \Z) \geq 3$, then $K$ cannot admit a symmetric union presentation.
\end{theorem}

We show that the ribbon Montesinos knot $11a_{103} = K(\frac{2}{3}, \frac{1}{3}, \frac{2}{7})$ is $\pi$-minimal in Proposition~\ref{prop:11a_{103}}.

%%%%%%%%%%%%
\subsection*{Acknowledgments}
This study was supported in part by JSPS KAKENHI Grant Numbers JP19K03505, JP23K12974, and JP24H00686.
The first and second authors were supported by Soka University International Collaborative Research Grant. 
The authors are grateful to the anonymous referee for crucial comments improving the manuscript.

%%%%%%%%%%%%%%%%%%%%%%%%%%%%%%%
\section{Preliminaries}
\label{sec:preliminaries}

According to \cite[Definition~3.2.2]{Kaw96}, a link $L \subset S^3$ is said to be \emph{locally trivial} if any $2$-sphere $S$ intersecting $L$ in two points bounds a $3$-ball $B$ such that $B \cap L$ is a trivial arc in $B$. 
The unknot $U$ is locally trivial.
For a split link, it is locally trivial if and only if it is a $2$-component trivial link $U\sqcup U$. 
A connected sum of two links distinct from the unknot is not locally trivial.

\begin{definition}[{\cite[Definition~3.2.4]{Kaw96}}]
\label{def:prime}
A link $L \subset S^3$ is \emph{prime} if it is locally trivial and neither the unknot nor $2$-component trivial link.
\end{definition}

The following factorization theorem will be useful (see \cite[Main Theorem]{Has58}, \cite[Theorem~3.2.6]{Kaw96}).

\begin{theorem}[Unique factorization theorem]\label{thm:uniquefactor} 
A non-split link other than the unknot decomposes into a connected sum of finitely many prime links. 
Furthermore, the prime factors are unique, up to a permutation. 
\end{theorem}

%%%
\subsection{Montesinos links}
\label{subsec:Montesinos}
First, we recall fundamental properties of a Montesinos link. 
Let $L = L(\frac{\beta_1}{\alpha_1}, \dots, \frac{\beta_r}{\alpha_r})$ be a Montesinos link with $r \geq 3$ rational tangles of slopes $\beta_i/\alpha_i \in \Q$ with $\alpha_i > 1$ and $\beta_i \neq 0$ coprime to $\alpha_i$. 
It has been shown by Montesinos~\cite{Mon73} that the $2$-fold branched cover $\Sigma_2(L)$ of $L$ is the closed orientable Seifert fibered $3$-manifold 
$V(0; e_0; \frac{\beta_1}{\alpha_1}, \dots, \frac{\beta_r}{\alpha_r})$ with base a $2$-dimensional orientable orbifold with underlying space $S^2$ and $r$ singular points with branching indices $\alpha_i \geq 2$, corresponding to the $r$ exceptional fibers of types $(\alpha_i, \beta_i)$.
See also \cite[Chapter~12.D]{BZH14}. 
Its rational Euler number $e_0 = \sum_{i=1}^{r} \frac{\beta_i}{\alpha_i} \in \Q$ satisfies 
\[
\det L = |H_1(\Sigma_2(L); \Z)| = |e_0| \prod_{i=1}^{r} \alpha_i.
\]
See \cite[Corollary~6.2]{JaNe83} for instance. 
Moreover, $r-2 \leq \rank \pi_1(\Sigma_2(L)) \leq r-1$ by \cite[Theorem~1.1]{BoZi84}. 

The Seifert fibered manifold $\Sigma_2(L) = V(0; e_0; \frac{\beta_1}{\alpha_1}, \dots, \frac{\beta_r}{\alpha_r})$ is determined, up to orientation-preserving homeomorphism, by the rational Euler number $e_0 \in \Q$ and the set of fractions $\{\frac{\beta_1}{\alpha_1}, \dots, \frac{\beta_r}{\alpha_r}\}$ in $\Q/\Z$ up to permutations (see \cite{Orl72}, \cite[Theorem~1.5]{JaNe83}), while the Montesinos link $L(\frac{\beta_1}{\alpha_1}, \dots, \frac{\beta_r}{\alpha_r})$ is determined, up to orientation-preserving homeomorphism of $S^3$, by the rational Euler number $e_0 \in \Q$ and the set of fractions $\{\frac{\beta_1}{\alpha_1}, \dots, \frac{\beta_r}{\alpha_r}\}$ in $\Q/\Z$ up to dihedral permutations (see \cite[Theorem~12.26]{BZH14}).
Note that the sign convention for the rational Euler number $e_0 = \sum_{i} \frac{\beta_i}{\alpha_i}$ is opposite to the choice made in \cite[Section~1]{NeRa78}.

%%%%
\subsection{Orbifolds}
Here we give some definitions and terminologies on orbifolds. 
An \emph{$n$-orbifold} is a metrizable space locally modeled on $\R^n$ modulo finite group actions. 
See \cite{BMP03} or \cite{CHK00} as a general reference.

\begin{definition}
An $n$-orbifold $\OO$ is said to be \emph{good} if it has some covering orbifold which is a manifold. 
Otherwise, it is \emph{bad}. 
Every orbifold $\OO$ has a universal covering $\pi\colon \widetilde{\OO}\to \OO$ with the property that for any covering $p\colon \OO' \to \OO$ 
there is a covering $\pi'\colon \widetilde{\OO}\to \OO'$ such that $\pi=p\circ \pi'$. 
The \emph{orbifold fundamental group} $\pi_1^{\orb}(\OO)$ of an orbifold $\OO$ is the group of deck transformations of the universal cover $\widetilde{\OO}\to \OO$.
\end{definition}

An orbifold is said to be \emph{compact} if its underlying topological space is compact. 
\begin{definition}
A compact $n$-orbifold $\OO$ is called 
\begin{enumerate}[label=(\arabic*)]
\item \emph{discal} if it is a finite quotient of the $n$-disk $D^n$ by an orthogonal action,
\item \emph{spherical} if it is a finite quotient of $S^n$ by an orthogonal action. 
\end{enumerate}
\end{definition}

A $k$-dimensional suborbifold of a $3$-orbifold is locally modeled on the inclusion $\R^{k}\subset \R^{3}$ modulo a finite group. 
The next definition extends the notion of irreducibility to the setting of $3$-orbifolds.

\begin{definition}
A compact $3$-orbifold $\OO$ is said to be \emph{irreducible} if it does not contain any bad $2$-suborbifold and if every orientable spherical $2$-suborbifold bounds a discal $3$-suborbifold. 
If $\OO$ is not irreducible, then it is \emph{reducible}. 
\end{definition}

There are only four bad compact $2$-orbifolds: first the teardrop $S^2(n)$ which is a $2$-sphere with a cone point of order $n \geq 2$, and its quotient $D^2(n)$ which is a $2$-disk with a corner mirror point of order $2n$; 
second the spindle $S^2(m, n)$ which is a $2$-sphere with two cone points of order $m$ and $n$ ($2\leq m<n$), and its quotient $D^2(m, n)$ which is a $2$-disk with two corner mirror points of order $2m$ and $2n$. 
The orbifold theorem implies that a compact orientable $3$-orbifold $\OO$ admits a finite cover which is a manifold if and only if $\OO$ does not contain a bad $2$-suborbifold, see \cite[Corollary~1.3]{BLP05}.

An orbifold is \emph{closed} if it is compact and has empty orbifold boundary. 
The notion of compressible, incompressible, and essential closed $2$-suborbifold can also be defined in the orbifold context.

\begin{definition}
\label{def:compressible}
Let $F$ be a closed connected $2$-suborbifold in a compact $3$-orbifold $\OO$. 
\begin{enumerate}[label=(\arabic*)]
\item $F$ is \emph{compressible} if either $F$ bounds a discal $3$-suborbifold in $\OO$ or there is a discal $2$-suborbifold $\Delta$ which intersects $F$ transversally in $\partial\Delta=\Delta\cap F$ and such that $\partial \Delta$ does not bound a discal $2$-suborbifold in $F$.
\item $F$ is \emph{incompressible} if no connected component of $F$ is compressible in $\OO$. 
\item $F$ is \emph{essential} if it is incompressible and not boundary parallel. 
\end{enumerate}
\end{definition}

Like for a closed $3$-manifold, we can now define a small closed $3$-orbifold as follows. 

\begin{definition}
A closed $3$-orbifold $\OO$ is said to be \emph{small} if it is irreducible, and it does not contain any essential embedded closed orientable $2$-suborbifold. 
\end{definition}

A Seifert fibered $3$-orbifold is defined as follows.
See \cite{BoSi85}, \cite[Chapter~5]{BMP03}, or \cite{CHK00}.

\begin{definition}
A \emph{Seifert fibration} on a $3$-orbifold $\OO$ is a partition of $\OO$ into closed $1$-suborbifolds called fibers, such that each fiber has a neighborhood modeled on $(S^1\times D^2)/G$, where $G$ is a finite group which acts on each factor. 
By collapsing each fiber to a point, an orbifold bundle projection onto a 2-orbifold can be obtained. 
If $\OO$ admits a Seifert fibration, then it is called \emph{Seifert fibered}.
The fibers are either circles or intervals with mirror endpoints.
\end{definition}

%%%
\subsection{$\pi$-domination}
\label{subsec:pi-domination}

We recall the definition of $\pi$-domination between links in $S^3$.
Let $L$ and $L'$ be links in $S^3$, $L$ \emph{$\pi$-dominates} $L'$, written $L \succeq L'$, if there is an epimorphism $\varphi\colon G^\orb(L) \twoheadrightarrow G^\orb(L')$ such that the following diagram is commutative:
\[
\xymatrix{
G^\orb(L) \ar@{->>}[r]^{\nu} \ar@{->>}[d]_-{\varphi} & \Z/2\Z \ar[d]^-{\id} \\
G^\orb(L') \ar@{->>}[r]^{\nu'} & \Z/2\Z.}
\]

Here, we show some basic properties of $\pi$-domination.

\begin{proposition}\label{prop:firstproperties}
Let $L$ be a link in $S^3$.
\begin{enumerate}[label=\textup{(\arabic*)}]
\item $L \succeq L'$ implies that $\# L \geq \# L'$, where $\# L$ denotes the number of components of $L$.
\item Any link $\pi$-dominates the unknot.
\item A connected sum $L_1 \sharp L_2$  $\pi$-dominates each summand $L_1$ and $L_2$. 
%\item A link $L$ $\pi$-dominates each sublink $L' \subset L$, in particular, each component.
\item A split link $L_1 \sqcup L_2$ $\pi$-dominates a connected sum $L_1 \sharp L_2$, and hence $L_1$ and $L_2$.

\end{enumerate}
\end{proposition}

\begin{proof}
For a link $L \subset S^3$, the abelianization $G^\orb(L)^{\ab}$ of the $\pi$-orbifold group is isomorphic to $(\Z/2\Z)^{\# L}$. 
Then, (1) follows from the fact that $\varphi$ induces an epimorphism between the abelianizations of the $\pi$-orbifold groups.
Moreover, (2) is immediate from the epimorphism $G^\orb(L) \xrightarrow{\nu} \Z/2\Z \cong G^\orb(U)$, where $U$ denotes the unknot.

(3) 
For a connected sum, we have $G^\orb(L_1 \sharp L_2) = G^\orb(L_1) \ast_{\Z/2\Z} G^\orb(L_2)$. 
Using the epimorphism $G^\orb(L_2) \xrightarrow{\nu} \Z/2\Z$, one gets an epimorphism $\varphi \colon G^\orb(L_1 \sharp L_2) \to G^\orb(L_1) \ast_{\Z/2\Z} \Z/2\Z = G^\orb(L_1)$, which fits into the diagram of Definition~\ref{def:domination}.

(4) 
There is an epimorphism 
\[
G^\orb(L_1 \sqcup L_2) = G^\orb(L_1) \ast G^\orb(L_2) \to G^\orb(L_1) \ast_{\Z/2\Z} G^\orb(L_2) = G^\orb(L_1 \sharp L_2)
\]
which amalgamates the two cyclic subgroups of order $2$ generated by the images of the meridians involved in the connected sum. 
This epimorphism fits into the diagram of Definition~\ref{def:domination}.
\end{proof}

The next lemma shows that the relation $\succeq$ is strongly related to the domination relation on the set of closed 
orientable $3$-manifolds defined via epimorphisms between their fundamental groups.

\begin{lemma}\label{lem:epicover}
Let $L$ and $L'$ be links in $S^3$.
If $L \succeq L'$ via an epimorphism $\varphi\colon G^\orb(L) \twoheadrightarrow G^\orb(L')$, then 
$\varphi$ induces an epimorphism $\tilde{\varphi}\colon \pi_1(\Sigma_2(L)) \twoheadrightarrow \pi_1(\Sigma_2(L'))$.
\end{lemma}

\begin{proof}
We have the commutative diagram
\[
\xymatrix{
1\ar[r] & \pi_1(\Sigma_2(L)) \ar[r] & G^\orb(L) \ar[r] \ar@{->>}[d]^-{\varphi} & \Z/2\Z \ar[d]^-{\id} \ar[r] & 1 \\
1\ar[r] & \pi_1(\Sigma_2(L')) \ar[r] & G^\orb(L') \ar[r] & \Z/2\Z \ar[r] & 1,}
\]
where the two rows are exact.
It follows from a diagram chasing that $\varphi$ induces a homomorphism $\tilde{\varphi}\colon \pi_1(\Sigma_2(L))\to \pi_1(\Sigma_2(L'))$ fitting into the above diagram.
Furthermore, $\tilde{\varphi}$ is surjective by a diagram chasing.
\end{proof}

%The next lemma shows that the commutativity of the diagram in Definition~\ref{def:domination} always holds for knots. 
The next lemma states that, for a given epimorphism between $\pi$-orbifold groups, if the domain arises from a knot, then the commutativity of the diagram in Definition~\ref{def:domination} always holds.
This remains true for links $L$ and $L'$ with the same number of components when $\Sigma_2(L')$ is aspherical.

\begin{lemma}
\label{lem:diagram}
\begin{enumerate}[label=\textup{(\arabic*)}]
\item If $K$ is a knot and there is an epimorphism $\varphi \colon G^\orb(K)\twoheadrightarrow G^\orb(L')$, then $L'$ is a knot and $K \succeq L'$. 
\item Suppose that $\# L = \# L'$ and that $\Sigma_2(L')$ is aspherical. If there is an epimorphism $\varphi \colon G^\orb(L)\twoheadrightarrow G^\orb(L')$, then $L \succeq L'$. 
\end{enumerate}
\end{lemma}

\begin{proof}
(1) 
Since $G^\orb(K)^{\ab} \cong \Z/2\Z$, $L'$ is a knot by the proof of Proposition~\ref{prop:firstproperties}(1). 
Let $\bar{\mu} \in G^\orb(K)$ be the image of any meridian $\mu \in G(K)$.
Since $\varphi(\bar{\mu})$ normally generates $G^\orb(L')$, $\varphi(\bar{\mu})$ does not belong to the normal subgroup $\pi_1(\Sigma_2(L'))$ of index $2$ in $G^\orb(L')$, and thus $\nu' \circ \varphi(\bar{\mu}) \neq 0$.

(2) 
The induced epimorphism $\bar{\varphi} \colon G^\orb(L)^{\ab} \to G^\orb(L')^{\ab}$ is an isomorphism since the abelian group $(\Z/2\Z)^{\# L}$ is hopfian. 
Then, for $\bar{\mu} \in G^\orb(L)$, $\varphi(\bar{\mu})$ is of order $2$. 
Hence, it does not belong to $\pi_1(\Sigma_2(L'))$ which is torsion-free by the hypothesis.
Therefore, $\varphi(\bar{\mu})$ maps to the generator of the quotient $G^\orb(L')/\pi_1(\Sigma_2(L')) \cong \Z/2\Z$.
\end{proof}

%%%%%%%%%%%%%%%%%%%%%%%%%%%%%%%%%%%
\section{Proof of Theorem~\ref{thm:order} concerning the preorder $\succeq$}
In this section, we give the proofs of Theorem~\ref{thm:order} and Corollary~\ref{cor:bridge_number}. 
Before the proofs, we recall the following facts about $2$-fold branched covers.

\begin{lemma}
\label{lem:folklore}
Let $L$ be a link in $S^3$.
\begin{enumerate}[label=\textup{(\arabic*)}]
\item $L$ is the unknot if and only if $\Sigma_2(L)\cong S^3$. 
\item If $L=L_1\sharp L_2$, then $\Sigma_2(L)\cong \Sigma_2(L_1)\sharp \Sigma_2(L_2)$.
\item $L$ is either a prime link or the unknot if and only if $\Sigma_2(L)$ is irreducible.
\end{enumerate}
\end{lemma}

(1) follows from the proof of the Smith conjecture (see \cite{MoBa84}).
(2) holds by construction.
(3) is proved by the equivariant sphere theorem (see \cite[Section~7]{MSY82} or \cite{Dun85}). 

\begin{proof}[Proof of Theorem~\ref{thm:order}]
Assume $L \succeq L'$ via an epimorphism $\varphi\colon G^\orb(L)  \twoheadrightarrow G^\orb(L')$. 
Then, by Lemma~\ref{lem:epicover}, $\varphi$ induces an epimorphism $\tilde{\varphi}\colon \pi_1(\Sigma_2(L)) \twoheadrightarrow \pi_1(\Sigma_2(L'))$. 
Now, the proof of Theorem~\ref{thm:order} uses the orbifold theorem, see \cite[Theorem~1]{BoPo01}.

\medskip
\noindent
(1)
If $L$ is the unknot, $G^\orb(L) \cong \Z/2\Z$. 
Therefore, $L'$ is a knot by Proposition~\ref{prop:firstproperties}(1) and $G^\orb(L') \cong \Z/2\Z$, since the $\pi$-orbifold group is never trivial. 
Then it is a consequence of the proof of the Smith conjecture that $L'$ is the unknot if and only if $G^\orb(L') \cong \Z/2\Z$, see \cite{MoBa84} and \cite[Proof of Proposition~3.2]{BoZi89}.

\medskip
\noindent
(2)
It follows from the orbifold theorem that a link $L$ is a $2$-bridge link if and only if $G^\orb(L) $ is a dihedral group, see \cite[Proposition~3.2]{BoZi89}.
Therefore, $G^\orb(L) $ is a dihedral group.
Since the image $\tilde{\varphi}(\pi_1(\Sigma_2(L))) = \pi_1(\Sigma_2(L'))$ is cyclic or trivial, $G^\orb(L')$ is either a dihedral group or $\Z/2\Z$.
It follows that $L'$ is either a $2$-bridge link or the unknot.

\medskip
\noindent
(3)
We distinguish two cases according to whether the $\pi$-orbifold group $G^\orb(L')$ is finite or not.
If $G^\orb(L')$ is \emph{finite}, then $\pi_1(\Sigma_2(L'))$ is finite.
By the orbifold theorem, $\Sigma_2(L')$ is a Seifert $3$-manifold with finite fundamental group and $L'$ is either the unknot, a $2$-bridge link, or an elliptic Montesinos link with three rational tangles. 

If $G^\orb(L')$ is \emph{infinite}, then $G^\orb(L)$ is infinite and $\Sigma_2(L)$ is a Seifert fibered $3$-manifold with an infinite fundamental group.
Now, $\pi_1(\Sigma_2(L))$ contains an infinite cyclic center $Z$ and the quotient $\Gamma = \pi_1(\Sigma_2(L))/Z$ is the orbifold fundamental group of the base $\mathcal{B}$, which is generated by torsion elements since the underlying space is $S^2$ by the definition of the Montesinos link $L$.
We distinguish two cases (I) and (II) according to whether $L'$ is a prime link or not.

(I) 
If $L'$ is prime, then $\Sigma_2(L')$ is irreducible by Lemma~\ref{lem:folklore}(3).
Since $\pi_1(\Sigma_2(L'))$ is infinite, $\Sigma_2(L')$ is an aspherical $3$-manifold (see \cite[Proposition~1.1(a)]{BoZi89}).
The epimorphism $\tilde{\varphi}\colon \pi_1(\Sigma_2(L)) \twoheadrightarrow \pi_1(\Sigma_2(L'))$ maps the infinite cyclic center $Z$ of $\pi_1(\Sigma_2(L))$ into the center of $\pi_1(\Sigma_2(L'))$.

Suppose, to the contrary, that $\pi_1(\Sigma_2(L'))$ is centerless. 
Then $\tilde{\varphi}(Z) =\{1\}$ and $\tilde{\varphi}$ induces an epimorphism from the orbifold group $\Gamma = \pi_1(\Sigma_2(L))/Z$ onto $\pi_1(\Sigma_2(L'))$.
Since $\Sigma_2(L')$ is aspherical, $\pi_1(\Sigma_2(L'))$ is torsion-free.
On the other hand, $\Gamma$ is generated by torsion elements.
Thus, the image $\tilde{\varphi}(\Gamma)$ must be trivial, and this is impossible.
Therefore, $\pi_1(\Sigma_2(L'))$ has a non-trivial center.
In particular, it contains an infinite cyclic normal subgroup and it follows from \cite[Theorem~1.1]{CaJu94} or \cite[Corollary~2]{Gab92} that 
$\Sigma_2(L')$ admits a Seifert fibration.
By the orbifold theorem (see \cite{MeSc86}), one can assume that the covering involution $\tau'$ on $\Sigma_2(L')$ is fiber-preserving. 
If $\tau'$ preserves the orientation of the fibers, $L'$ is a Seifert link. 
Otherwise, the underlying space of the base of the Seifert fibration is $S^2$ by \cite[Section~5.11]{Mon73} and $L'$ is a Montesinos link with $r'$ rational tangles by \cite[Section~2]{Mon73}, where
\[
r' \leq \rank \pi_1(\Sigma_2(L')) +2 \leq \rank \pi_1(\Sigma_2(L)) +2 \leq r + 1.
\]

The center $Z = \ang{h}$ of $\pi_1(\Sigma_2(L))$ is an infinite cyclic normal subgroup in $G^\orb(L)$ generated by an element $h$ such that $\bar{\mu} h \bar{\mu}^{-1} = h^{-1}$, where $\bar{\mu}$ is the image of a meridian $\mu$ of $G(L)$ in the $\pi$-orbifold group $G^\orb(L)$.
This follows from the fact that, for a Montesinos link, the covering involution $\tau$ on $\Sigma_2(L)$ is fiber-preserving and reverses the orientation of the fibers (see \cite[Section~2]{Mon73} and \cite[Theorem~1.3]{BoZi89}).
The image $\varphi (Z)$ of the center $Z$ of $\pi_1(\Sigma_2(L))$ is in the center of $\pi_1(\Sigma_2(L'))$.

Suppose, to the contrary, that the link $L'$ is a Seifert link.
Then the covering involution $\tau'$ preserves the orientation of the fibers of $\Sigma_2(L')$, which means that $\varphi(Z)$ belongs to the center of the $\pi$-orbifold group 
$G^\orb(L')$ (see \cite[Corollary~1.4]{BoZi89}).
Therefore,
\[
 \varphi(h) = \varphi(\bar{\mu}) \varphi(h) \varphi(\bar{\mu})^{-1} = \varphi(\bar{\mu} h \bar{\mu}^{-1}) = \varphi(h)^{-1}.
\]
This implies that $\varphi(h) = \varphi(h)^{-1}$, which is impossible because $\varphi(h)$ is a non-trivial element of infinite order in $\pi_1(\Sigma_2(L')) \subset G^\orb(L')$ since $\Sigma_2(L')$ is aspherical.

(II) 
Assume that $L'$ is not prime and has at most three bridges.
First, we consider the case where $L'$ is split: $L' = L'_1 \sqcup L'_2$. 
Then $L'$ cannot be the $2$-component trivial link $U\sqcup U$ which is a $2$-bridge link. 
It can always be expressed as a non-trivial connected sum $L'_1 \sharp (U \sqcup U) \sharp L'_2$. 
It follows from Lemma~\ref{lem:folklore}(1) and (2) that $\Sigma_2(L') = \Sigma_2(L'_1) \sharp (S^1 \times S^2) \sharp \Sigma_2(L'_2)$, with $\Sigma_2(L'_1)$ or $\Sigma_2(L'_2)$ not being $S^3$. 
Therefore, $\pi_1(\Sigma_2(L')) = \pi_1(\Sigma_2(L'_1)) \ast \Z \ast \pi_1(\Sigma_2(L'_2))$ is a non-trivial free product, and hence centerless.
The epimorphism $\tilde{\varphi}\colon \pi_1(\Sigma_2(L)) \twoheadrightarrow \pi_1(\Sigma_2(L'))$ must kill the infinite cyclic center $Z$ of $\pi_1(\Sigma_2(L))$.
Thus $\tilde{\varphi}$ factors through $\Gamma = \pi_1(\Sigma_2(L))/Z$ and induces an epimorphism onto each factor of $\pi_1(\Sigma_2(L'_1)) \ast \Z \ast \pi_1(\Sigma_2(L'_2))$.
As mentioned at the beginning of (3), $\Gamma$ is generated by torsion elements, and thus each factor must contain some torsion element. 
Therefore, $L$ cannot be split. 

We now assume that $L'$ is not split. 
Since $L'$ is not prime by Theorem~\ref{thm:uniquefactor}, $L'$ decomposes into a connected sum $L'=L'_1\sharp\cdots\sharp L'_n$ of finitely many prime links which are unique up to permutations ($n \geq 2$).
Then, by Lemma~\ref{lem:folklore}(1)--(3), $\Sigma_2(L') \cong \Sigma_2(L'_1) \sharp\cdots\sharp \Sigma_2(L'_n)$ is a non-trivial connected sum of closed, orientable, irreducible $3$-manifolds. 
In particular, $\pi_1(\Sigma_2(L')) = \pi_1(\Sigma_2(L'_1))  \ast \cdots \ast \pi_1(\Sigma_2(L'_n))$ is a non-trivial free product and hence centerless. 
The argument given above for the split case shows that each factor $\pi_1(\Sigma_2(L'_i))$ must contain some torsion element ($1 \leq i \leq n$). 
Since $\pi_1(\Sigma_2(L'_i))$ is the fundamental group of a closed irreducible orientable $3$-manifold for $1 \leq i \leq n$, it must be of finite order by \cite[Corollary~9.9]{Hem04}. 
By the orbifold theorem, each $\Sigma_2(L'_i)$ is a Seifert $3$-manifold with finite fundamental group, and each link $L'_i$ is either a $2$-bridge link or an elliptic Montesinos link with three rational tangles.
Let $n_1$ be the number of 2-bridge factors of $L'$ and $n_2$ the number of elliptic Montesinos factors of $L'$.
Then, the number $n = n_1 + n_2$ of prime factors of $L'$ satisfies $n_1 + 2n_2 \leq r-1$ since the fundamental group of the $2$-fold branched cover of a $2$-bridge link has rank $1$, that of an elliptic Montesinos link has rank $2$, and that of a Montesinos link with $r\geq 3$ rational tangles has rank at most $r-1$.

\medskip
\noindent
(4)
Let $L$ be a Seifert link with $\det L \neq 0$. 
Then $\Sigma_2(L)$ is a rational homology $3$-sphere (i.e., closed $3$-manifold with $H_\ast(\Sigma_2(L);\Q)\cong H_\ast(S^3;\Q)$) which is Seifert fibered. 
Like in the case of Montesinos links, we distinguish two cases according to whether the $\pi$-orbifold group $G^\orb(L')$ is finite or not.

If $G^\orb(L')$ is finite, then $\pi_1(\Sigma_2(L'))$ is finite and, by the orbifold theorem, $\Sigma_2(L')$ is a Seifert $3$-manifold with finite fundamental group and 
$L'$ is either the unknot, a $2$-bridge knot, or an elliptic Montesinos knot with three rational tangles.

If $G^\orb(L')$ is infinite, then $G^\orb(L)$ is infinite and $\Sigma_2(L)$ is a Seifert fibered $3$-manifold with infinite fundamental group. 
The fundamental group $\pi_1(\Sigma_2(L))$ contains an infinite cyclic center $Z$, and since $\Sigma_2(L)$ is a rational homology sphere, the underlying space of its base is $S^2$ or the projective plane $P^2$. 
It follows that the quotient $\Gamma = \pi_1(\Sigma_2(L))/Z$ is generated by torsion elements.
Then, like in the proof of the previous case (3) for Montesinos links, we distinguish two cases according to whether the link $L'$ is prime or not.

If the link $L'$ is prime, then the proof given in the previous case (3) shows that $\Sigma_2(L')$ is a Seifert manifold and that the epimorphism $\tilde{\varphi}\colon \pi_1(\Sigma_2(L)) \twoheadrightarrow \pi_1(\Sigma_2(L'))$ maps the infinite cyclic center $Z$ of $\pi_1(\Sigma_2(L))$ into a non-trivial subgroup of the center of $\pi_1(\Sigma_2(L'))$. 
Moreover, $Z$ belongs to the center of the $\pi$-orbifold group $G^\orb(L)$, see \cite[Corollary~1.4]{BoZi89}. 
Then the non-trivial subgroup $\varphi(Z)$ belongs to the center of the $\pi$-orbifold group $G^\orb(L')$. 
Therefore, $L'$ is a Seifert link by \cite[Corollary~1.4]{BoZi89}.

If $L'$ is not prime, then the proof given in the previous case (3) shows that $L'$ must be a connected sum of $n_1$ $2$-bridge knots and $n_2$ elliptic Montesinos knots with $n_1 + 2n_2 \leq \rank \pi_1(\Sigma_2(L))$. 
\end{proof}

\begin{proof}[Proof of Corollary~\ref{cor:bridge_number}] 
By \cite{BoZi85}, the bridge number $b(K)$ of a Montesinos knot with $r \geq 3$ rational tangles is $r$. 
Then, the proof follows readily from Theorem~\ref{thm:order}(3) when $K'$ is prime. 
When $K'$ is not prime, it is the connected sum of $n_1$ $2$-bridge knots and $n_2$ elliptic Montesinos knots with $n_1 + 2n_2 \leq r-1$. 
Schubert's bridge formula \cite{Sch54} shows that the bridge number minus one is additive by connected sum: 
\[
b(K_1 \sharp K_2) - 1 = (b(K_1)-1) + (b(K_2)-1). 
\]
It follows that the bridge number of the connected sum of $n_1$ $2$-bridge knots and $n_2$ elliptic Montesinos knots is $n_1 + 2n_2 + 1 \leq r = b(K)$.
\end{proof}

%%%%%%%%%%%%%%%%%%%%%%%%
\section{Proof of Theorem~\ref{thm:arb} concerning arborescent links}

Let $N$ be a closed orientable irreducible $3$-manifold whose JSJ pieces are all hyperbolic. 
Then $N$ is aspherical and thus $\pi_1(N)$ is torsion-free. 
Moreover, by \cite[Theorem~VI.1.6(i)]{JaSh79}, the centralizer in $\pi_1(N)$ of any non-trivial element is isomorphic to $\Z$ or $\Z \oplus \Z$ (see also \cite{Fri11}), hence abelian. 
Then, Theorem~\ref{thm:arb} follows from  Theorem~\ref{thm:graphtohyperbolic} and Corollary~\ref{cor:noepimorphism}, together with Lemma~\ref{lem:epicover}. 
The proof of Theorem~\ref{thm:graphtohyperbolic} is based on Lemmas~\ref{lem:trivialseifert} and \ref{lem:graphdecomposition} below.

\begin{lemma}\label{lem:trivialseifert}
Let $M$ be a rational homology $3$-sphere which is Seifert fibered with an orientable base. 
Let $G$ be a torsion-free group such that the centralizer in $G$ of any non-trivial element is abelian.
Then, any homomorphism $\phi \colon \pi_1(M) \to G$ is trivial.
\end{lemma}

\begin{proof}
By the assumption, the base of the Seifert manifold $M$ is a $2$-sphere with finitely many singular points. 
If $\pi_1(M)$ is finite, then the lemma follows from the fact that the group $G$ is torsion-free. 
So we assume that $\pi_1(M)$ is infinite. 
Since the base is orientable, the regular fiber $h \in \pi_1(M)$ generates a central infinite cyclic group of $\pi_1(M)$ and the quotient $\Gamma = \pi_1(M)/\ang{h}$ is generated by torsion elements.

Let $\phi \colon \pi_1(M) \to G$ be a homomorphism. 
Suppose, to the contrary, that $\phi(h) \neq 1$.
Then $\phi(h)$ generates a central infinite cyclic subgroup of the image $\phi(\pi_1(M))$ since $G$ is torsion-free. 
It follows that $\phi(\pi_1(M))$ is a subgroup of the centralizer $C_{G}(\phi(h))$ of $\phi(h)$ in $G$.
By the assumption, $C_{G}(\phi(h))$ is torsion-free and abelian, and so is $\phi(\pi_1(M))$.
This contradicts the fact that the abelianization of $\phi(\pi_1(M))$ is finite since $M$ is a rational homology sphere. 
Therefore, $\phi(h) = 1$ and $\phi$ induces a homomorphism $\Gamma \to G$.
Since $G$ is torsion-free and $\Gamma$ is generated by torsion elements, the induced homomorphism is trivial.
\end{proof}

\begin{definition}
A compact Seifert fibered $3$-manifold is said to be \emph{totally orientable} if it is orientable with an orientable base orbifold. 
A \emph{totally orientable graph structure} on a compact orientable prime $3$-manifold $M$ is a decomposition of $M$ along essential, non-parallel tori into totally orientable Seifert fibered submanifolds such that the adjacent Seifert fibrations do not match on the splitting tori.
\end{definition}

\begin{lemma}\label{lem:graphdecomposition} 
Let $M$ be a compact orientable prime graph manifold. 
Then $M$ admits a totally orientable graph structure.
\end{lemma}

\begin{proof}
Let $M$ be a prime graph manifold. 
By definition, $M$ admits a decomposition along essential, non-parallel tori into Seifert fibered submanifolds such that the adjacent Seifert fibrations do not match on the splitting tori.
If the base of a Seifert piece $S$ is non-orientable, then the underlying space is a punctured non-orientable surface.
A closed non-orientable surface is a connected sum of $k$ projective planes for some $k \geq 1$. 
Then, a non-orientable surface with $q\geq 1$ punctures is either a M\"{o}bius band if $k = q =1$ or it is obtained by gluing $k$ M\"{o}bius bands along some boundary components of a planar surface with $k + q \geq 3$ boundary components.
The Seifert fibration over each M\"{o}bius band gives a Seifert fibered submanifold $Q_i$ ($i=1, \dots, k$) of $S$ which is a twisted $S^1$-bundle over a M\"{o}bius band having an incompressible torus boundary. 
Then one can consider, on each $Q_i$, the other Seifert fibration whose base is a disk with two singular points of order $2$. 

If the base of $S$ is a M\"{o}bius band, then this operation replaces the Seifert piece $S$ with a non-orientable base by a Seifert piece with an orientable base. 
Otherwise, it replaces $S$ by the union of $k+1$ Seifert pieces, which are $S \setminus \bigsqcup_{i} \Int(Q_i)$ and the $k$ Seifert manifolds $Q_i$. 
Here, $S \setminus \bigsqcup_{i} \Int(Q_i)$ has a planar base and each $Q_i$ has an orientable base and an incompressible boundary.
Moreover, the Seifert piece $S \setminus \bigsqcup_{i} \Int(Q_i)$ is not homeomorphic to $T^2 \times [0, 1]$ because the underlying surface of its base has at least three boundary components.
In the first case, the new Seifert fibration on $S$ may match the Seifert fibration of the adjacent Seifert piece, which gives a single Seifert piece. 
In the second case, the Seifert fibrations do not match any more on the gluing tori $\partial Q_i$, and therefore one gets $k+1$ distinct Seifert pieces. 
\end{proof}

The graph dual to this totally orientable graph structure is obtained from the graph dual to the JSJ decomposition in the following way: if the underlying space of the base of $S$ is a M\"{o}bius band, then either the vertex associated to $S$ remains unchanged or it is crashed with its edge into the adjacent vertex. 
Otherwise, $k \geq 1$ extra leaf vertices associated to $Q_i$'s are connected by $k$ edges to the vertex associated to $S \setminus \bigsqcup_{i} \Int(Q_i)$ (previously associated to $S$).

\begin{remark}\label{rem:planarbase} 
For a rational homology $3$-sphere $M$ which is a graph manifold, the bases of the Seifert fibered pieces have planar underlying spaces if the bases are orientable. 
Moreover, the graph dual to the JSJ decomposition, which associates a vertex to each Seifert piece and an edge to each torus, is a tree.
\end{remark}

The following is a straightforward corollary of Lemma~\ref{lem:graphdecomposition} and Remark~\ref{rem:planarbase}.

\begin{corollary}\label{cor:planarbase} 
Let $M$ be a rational homology $3$-sphere which is a prime graph manifold. 
Then, $M$ admits a totally orientable graph structure such that the bases of the Seifert pieces have planar underlying spaces and the graph dual to the decomposition is a tree.
\end{corollary}

\begin{remark}\label{rem:mod2homology} 
When $M$ is the $2$-fold branched cover $\Sigma_2(K)$ of a knot $K \subset S^3$, the bases of the Seifert fibered JSJ pieces of $M$ have planar underlying spaces since $M$ is a $\Z/2\Z$-homology sphere.
If the base of a Seifert fibered JSJ piece is non-orientable, then the orientation covering of the base induces a connected $2$-fold cover of $M$, which is not possible.
\end{remark}

\begin{theorem}\label{thm:graphtohyperbolic}
Let $M$ be a rational homology $3$-sphere which is a graph manifold, and let $G$ be a torsion-free group where the centralizer in $G$ of any element is abelian.
Then any homomorphism $\phi \colon \pi_1(M) \to G$ is trivial.
\end{theorem}

\begin{proof}
Let $M$ be a rational homology $3$-sphere which is a graph manifold.
If $M$ is not prime, then $M = M_1 \sharp \cdots \sharp M_n$ is a connected sum of prime graph manifolds which are rational homology spheres, and hence irreducible. 
Since $\pi_1(M)$ is isomorphic to the free product $\pi_1(M_1) \ast\cdots \ast \pi_1(M_n)$, one needs only to prove Theorem~\ref{thm:graphtohyperbolic} for a prime graph manifold. 
So, assume that $M$ is prime.
By Corollary~\ref{cor:planarbase}, we can always assume that the rational homology sphere $M$ admits a totally orientable graph structure and carry out an induction argument on the number of Seifert pieces for such a decomposition. 
The case of a single Seifert piece follows from Lemma~\ref{lem:trivialseifert}.

Let us assume that the statement is true when the manifold $M$ has a totally orientable graph structure with at most $n$ Seifert pieces and prove it when $M$ has such a decomposition with $n+1$ Seifert pieces.
Let $S \subset M$ be a Seifert piece of the decomposition corresponding to a leaf in the tree dual to the decomposition. 
Then, $S$ is a Seifert $3$-manifold with boundary a torus and base a disk with finitely many singular points. 
The regular fiber $h \in \pi_1(S)$ generates the center of $\pi_1(S)$ and the quotient $\Gamma_S = \pi_1(S)/\ang{h}$ is generated by torsion elements. 

Let $\phi \colon \pi_1(M) \to G$ be a homomorphism. 
Suppose, to the contrary, that $\phi(h) \neq 1$.
Now, $\phi(\pi_1(S))$ is torsion-free and abelian because it is a subgroup of the centralizer $C_G(\phi(h))$. 
Since $S$ is a rational homology solid torus, $H_1(S; \Q) \cong \Q$ and so $\phi(\pi_1(S)) \cong \Z$.
In particular, $\phi(\pi_1(\partial S)) \cong \Z$ since $h \in \pi_1(\partial S)$. 
Therefore, there is a simple closed curve $\gamma \subset \partial S$ such that $\phi(\gamma)$ is trivial.

Let $M_0 = \overline{M \setminus S}$ be the closure of the complement of $S$ in $M$. 
It is a graph manifold whose boundary $\partial S$ is an incompressible torus, and it admits a totally orientable graph structure with $n$ Seifert pieces. 
Let $M_0(\gamma) = M_0 \cup V$ be the Dehn filling obtained by gluing a solid torus $V$ to $\partial M_0$ in such a way that the boundary of the meridian disk of $V$ is identified with the curve $\gamma$. 
In the same way, we define the Dehn filling $S(\gamma) = S\cup V$. 
Then the homomorphism $\phi \colon \pi_1(M) \to G$ factors through a ``squeeze'' along the incompressible torus $\partial M_0 = \partial S$ which kills the simple closed curve $\gamma \subset \partial S$ and induces a homomorphism 
\[
\phi' \colon \pi_1(M_0(\gamma)) \ast_{\Z} \pi_1(S(\gamma)) \to G, 
\] 
where the amalgamating group $\Z$ is generated by a simple closed curve $\delta \subset \partial S$ dual to $\gamma$.

\begin{claim}\label{claim:rationalhomology}
The closed manifold $M_0(\gamma)$ is a rational homology sphere.
\end{claim}

\begin{proof}
The Mayer-Vietoris exact sequence for the rational homology sphere $M = M_0 \cup_{T^2} S$ implies that the homomorphism
\[
i_{\ast} \oplus j_{\ast}\colon H_1(T^2; \Q) \to H_1(M_0; \Q) \oplus H_1(S; \Q)
\]
induced by the inclusions $i\colon T^2 \to M_0$ and $j\colon T^2 \to S$ is an isomorphism. 
Since the isomorphism $\phi_{\ast}\colon H_1(S; \Q) \to \Q$ kills the element $j_{\ast}(\gamma) \in H_1(S; \Q)$, one has $j_{\ast}(\gamma)= 0$, and hence $i_{\ast}(\gamma)\neq 0 \in H_1(M_0; \Q)$. 
Then the homomorphism 
\[
\psi\colon H_1(T^2; \Q) \to H_1(M_0; \Q) \oplus H_1(V; \Q)
\]
in the Mayer-Vietoris exact sequence for the manifold $M_0(\gamma) = M_0 \cup_{T^2} V$ is an isomorphism. 
Therefore, $H_1(M_0(\gamma); \Q)= 0$ and $M_0(\gamma)$ is a rational homology sphere. 
\end{proof}

Since $M_0(\gamma)$ is a rational homology sphere, Lemma~\ref{lem:graphfilling} below implies that $\pi_1(M_0(\gamma))$ is a free product of the fundamental groups of rational homology spheres which are graph manifolds having totally orientable graph structures with at most $n$ Seifert pieces.
The induction hypothesis implies that the homomorphism $\phi'$ is trivial on each factor of this free product and hence $\phi'(\pi_1(M_0(\gamma))) =\{1\}$.
The assumption $\phi(h) \neq 1$ implies that $\phi(\pi_1(S)) = \phi'(\pi_1(S(\gamma))) \cong \Z$ and that
\[
\phi(\pi_1(M)) = \phi' \bigl( \pi_1(M_0(\gamma)) \ast_{\Z} \pi_1(S(\gamma)) \bigr) = \phi'(\pi_1(S(\gamma))).
\]
This contradicts the fact that $M$ is a rational homology sphere, and therefore $\phi(h) = 1$.

Now, $\phi$ factors through $\Gamma_S$ which is generated by torsion elements.
Since $G$ is torsion-free, $\phi(\pi_1(S)) = \{1\}$.
It follows that $\phi$ factors through a homomorphism 
\[
\phi_0 \colon \pi_1(M_0)/\aang{\pi_1(\partial M_0)} \to G.
\] 
Since $M_0$ is a rational homology solid torus, one can choose a simple closed curve $\alpha \subset \partial M_0$ such that the Dehn filling $M_0(\alpha)$ is a rational homology sphere (see \cite[Lemma~3.2]{Wat12} for example). 
The natural epimorphism
\[
\pi\colon \pi_1(M_0(\alpha)) \twoheadrightarrow \pi_1(M_0)/\aang{\pi_1(\partial M_0)}
\] 
induces a homomorphism $\phi_0 \circ \pi \colon \pi_1(M_0(\alpha)) \to G$ and we have
\[
\phi_0 \circ \pi (\pi_1(M_0(\alpha))) = \phi_0(\pi_1(M_0)/\aang{\pi_1(\partial M_0)}) = \phi(\pi_1(M)).
\]
By the choice of $\alpha$, the manifold $M_0(\alpha)$ is a rational homology sphere and it follows from Lemma~\ref{lem:graphfilling} below that $M_0(\alpha)$ is a connected sum of closed, irreducible, rational homology spheres which are graph manifolds having totally orientable graph structures with at most $n$ Seifert pieces. 
The induction hypothesis, as above, implies $\phi_0 \circ \pi (\pi_1(M_0(\alpha))) = \{1\}$, and hence $\phi$ is trivial.
\end{proof}

\begin{lemma}\label{lem:graphfilling} 
Let $M_0$ be an orientable, irreducible graph manifold whose boundary is an incompressible torus and which admits a totally orientable graph structure with at most $n$ Seifert pieces.
If $M_0(\gamma)$ is a rational homology sphere, then $M_0(\gamma)$ is a connected sum of closed, irreducible, rational homology spheres which are graph manifolds having totally orientable graph structure with at most $n$ Seifert pieces.
\end{lemma}

\begin{proof}
Since $M_0(\gamma)$ is a rational homology sphere, $M_0$ is a rational homology solid torus. 
Therefore, the Seifert fibered pieces in a totally orientable graph structure on $M_0$ have planar underlying spaces. 
Then the proof is by induction on the number $n$ of Seifert pieces of a totally orientable graph structure on $M$.
If such a decomposition exists with a single Seifert piece, then the base of $M_0$ is a disk with $k$ cone points because $M_0$ has incompressible boundary, where $k\geq 2$. 
There are two cases according to whether the curve $\gamma$ on $\partial M_0$ is homotopic to the fiber of the Seifert fibration of $M_0$ or not.

If $\gamma$ is parallel to the regular fiber of $M_0$, then an essential properly embedded vertical annulus in $M_0$ will give an essential embedded $2$-sphere in $M_0(\gamma)$. 
Therefore, in this case, $M_0(\gamma)$ is a connected sum of $k$ lens spaces corresponding to the essential annuli surrounding the singular fibers. 
Hence $M_0(\gamma)$ is a connected sum of irreducible, rational homology spheres which are graph manifolds, and each one has a single Seifert piece with orientable base.

Consider the other case where the intersection number of $\gamma$ with the regular fiber of $M_0$ is $\delta \geq 1$.
Then, the Seifert fibration on $M_0$ extends to a Seifert fibration on $M_0(\gamma)$ over $S^2$ with $k$ or $k+1$ singular fibers depending on whether $\delta = 1$ or $\delta \geq 2$. 
Hence, $M_0(\gamma)$ is a graph manifold with a single Seifert piece whose base is orientable.

By induction, we assume that Lemma~\ref{lem:graphfilling} is true when a graph manifold has a totally orientable graph structure with at most $n$ Seifert pieces.
Let us prove it for $M_0$ having such a decomposition with $n+1$ Seifert pieces.
Consider the Seifert piece $Y$ of $M_0$ such that $\partial Y \supset \partial M_0$.

We first assume that $Y$ is a cable space (i.e., Seifert fibered over an annulus with a single cone point) and the intersection number $\delta$ of $\gamma$ with a regular fiber of $Y$ on $\partial M_0$ is $1$.
Then, $Y(\gamma)$ is a solid torus. 
Therefore, $M_0(\gamma) = \overline{M_0\setminus Y} \cup Y(\gamma)$ is a Dehn filling of the irreducible graph manifold $\overline{M_0\setminus Y}$ which admits a totally orientable graph structure with $n$ Seifert pieces. 
Then, the induction hypothesis implies that $M_0(\gamma)$ is a connected sum of closed, irreducible, rational homology spheres which are graph manifolds having totally orientable graph structures with at most $n$ Seifert pieces.

Next, we assume that $Y$ is a cable space and $\delta\geq 2$.
Then, the Seifert fibration on $Y$ extends to that on $Y(\gamma)$ over a disk with two exceptional fibers. 
Hence, $Y(\gamma)$ has an incompressible boundary. 
We conclude that $M_0(\gamma)$ is an irreducible graph manifold which admits a totally orientable graph structure with $n+1$ Seifert pieces since $\overline{M_0\setminus Y}$ has such a decomposition with $n$ Seifert pieces.

Then we turn to the case where $Y$ is not a cable space and $\gamma$ is not parallel on $\partial M_0$ to the regular fiber of $Y$.
The Seifert fibration on $Y$ extends to a Seifert fibration on the Dehn filling $Y(\gamma) = Y\cup V$ over an orientable base, where the base cannot be a disk with at most one cone point. 
Therefore, $Y(\gamma)$ is a Seifert manifold with an incompressible boundary. 
As in the previous case, $M_0(\gamma) = \overline{M_0\setminus Y} \cup Y(\gamma)$ is an irreducible graph manifold which admits a totally orientable graph structure with $n+1$ Seifert pieces.

Finally, we assume that $\gamma$ is parallel on $\partial M_0$ to the regular fiber of $Y$. 
Since $Y$ is irreducible, any essential sphere in $Y(\gamma)$ comes from an essential properly embedded planar surface in $Y$ with boundary curves parallel to $\gamma$ in $\partial M_0$. 
Such a surface must be isotopic to a vertical annulus which is a union of fibers. 
Then a maximal collection of disjoint, non-parallel, essential vertical annuli, properly embedded in $Y$ with boundaries in $\partial M_0$ gives a collection of essential, non-parallel, embedded $2$-spheres in $Y(\gamma)$ which splits $Y(\gamma)$ as a connected sum of prime Seifert manifolds $S_i$ with orientable bases. 
Since $M_0(\gamma)$ is a rational homology sphere, none of these prime Seifert manifolds $S_i$ can be homeomorphic to $S^1 \times S^2$ and so they are irreducible.

Since the underlying space of the base of $Y$ is planar, an essential annulus in such a maximal collection surrounds either a single singular fiber or a single boundary component of $\partial Y \setminus \partial M_0$. 
In the first case, the associated irreducible summand $S_i$ of $Y(\gamma)$ is a lens space, while in the second case it is a Seifert manifold with a single boundary component. 
Therefore, $M_0(\gamma)$ is a connected sum of finitely many lens spaces and finitely many closed manifolds obtained by gluing a connected component $N_i$ of $\overline{M_0\setminus Y}$ and one of the Seifert manifolds, say $S_i$, along a torus component in $\partial Y \setminus \partial M_0 = \partial \overline{M_0\setminus Y}$.
The graph manifold $N_i$ admits a totally orientable graph structure with at most $n$ Seifert pieces. 

Hence, if $\partial S_i$ is incompressible, then $N_i \cup S_i$ is an irreducible rational homology sphere which is a graph manifold having a totally orientable graph structure with at most $n+1$ Seifert pieces. 
If $\partial S_i$ is compressible, then $S_i$ is a solid torus and the induction hypothesis implies that $N_i \cup S_i$ is a connected sum of closed, irreducible, rational homology spheres which are graph manifolds having totally orientable graph structures with at most $n$ Seifert pieces. 
In summary, the discussion above shows that $M_0(\gamma)$ is a connected sum of closed, irreducible, rational homology spheres which are graph manifolds having totally orientable graph structures with at most $n+1$ Seifert pieces.
\end{proof}

Combining Theorem~\ref{thm:graphtohyperbolic} with \cite[Theorem~VI.1.6(i)]{JaSh79}, one has the next corollary.

\begin{corollary}\label{cor:noepimorphism}
The fundamental group of a rational homology $3$-sphere which is a graph manifold cannot surject onto the fundamental group of a closed orientable irreducible $3$-manifold whose JSJ pieces are all hyperbolic.
\end{corollary}

Using Corollary~\ref{cor:noepimorphism} and Lemma~\ref{lem:epicover}, we prove Theorem~\ref{thm:arb} concerning an arborescent link $L$ with $\det L \neq 0$.

\begin{proof}[Proof of Theorem~\ref{thm:arb}]
Let $L$ be an arborescent link with $\det L \neq 0$ and suppose $L \succeq L'$. 
If $L'$ is not prime, then it is a connected sum $L_1 \sharp L_2$ or a split union $L_1 \sqcup L_2$, and thus we have $L \succeq L_1$ and $L \succeq L_2$. 
It follows from Lemma~\ref{lem:folklore}(2) that $\Sigma_2(L_1 \sharp L_2) = \Sigma_2(L_1) \sharp \Sigma_2(L_2)$ and $\Sigma_2(L_1 \sqcup L_2) = \Sigma_2(L_1) \sharp (S^1 \times S^2) \sharp \Sigma_2(L_2)$. 
Therefore, it suffices to prove Theorem~\ref{thm:arb} when $L'$ is prime, which is equivalent to $\Sigma_2(L')$ being irreducible by Lemma~\ref{lem:folklore}(3).
By Lemma~\ref{lem:epicover}, the epimorphism $\varphi\colon G^\orb(L) \twoheadrightarrow G^\orb(L')$ induces an epimorphism $\tilde{\varphi}\colon \pi_1(\Sigma_2(L)) \twoheadrightarrow \pi_1(\Sigma_2(L'))$. 
Since $\Sigma_2(L)$ is a graph manifold which is a rational homology sphere and $\Sigma_2(L')$ is irreducible, Corollary~\ref{cor:noepimorphism} implies that a  JSJ piece of $\Sigma_2(L')$ must be Seifert fibered.
\end{proof}

%%%%%%%%%%%%%%%%%
\section{Small links and finiteness}
\label{sec:small}

In this section, we prove Theorem~\ref{thm:small}. 
Unless otherwise stated, surfaces are assumed to be orientable. 
We recall that a compact properly embedded surface $F$ in a link exterior $E(L)$ is \emph{essential} if it is incompressible and no component of $F$ is parallel to a subsurface of $\partial E(L)$ (cf.~Definition~\ref{def:compressible}).
A link $L \subset S^3$ is said to be \emph{small} if its exterior $E(L)$ does not contain any compact, properly embedded, essential surface whose boundary is empty or a collection of meridian curves. 
A small link is prime or the unknot.

If $L$ is not a split link, $E(L)$ is irreducible and a properly embedded essential surface $F$ in $E(L)$ is boundary-incompressible, see \cite[Lemma~2.1]{Oer84}.
If $\Sigma_2(L)$ does not contain any incompressible surface, then the link $L$ is small by \cite[Theorem~1]{GoLi84}. 
The converse is not true. 

\begin{remark}
Frequently, a link is said to be \emph{small} if there is no essential closed surface and \emph{meridionally small} if there is no essential surface in its exterior with non-empty boundary whose boundary forms parallel copies of the meridian. 
We note that our definition of small requires both of the above small and meridionally small properties.
In the knot case, if $K$ is small, then it is meridionally small by \cite[Theorem~2.0.3]{CGLS87}. 
\end{remark}

%%%%%
\subsection{Proof of Theorem~\ref{thm:small}(1)}
The following lemma is crucial for proving the assertion (1).

\begin{lemma}\label{lem:small}
A link $L \subset S^3$ is small if and only if the orbifold $\OO(L)$ is small. 
\end{lemma}

\begin{proof}
Let $L$ be a link such that the orbifold $\OO(L)$ is not small.
We first consider the case where $\OO(L)$ is reducible, that is, there is an essential elliptic 2-orbifold $S$ embedded in $\OO(L)$.
If $S\cap L =\emptyset$, then $S \subset E(L)$ does not bound a ball in $E(L)$. 
Therefore, $S$ splits $S^3$ in two balls $B_1$ and $B_2$ such that $B_i \cap L \neq \emptyset$ and $L$ is a split link, which is not a small link.
If $S \cap L \neq \emptyset$, then $S = S^2(2,2)$ is a football with two singular points whose branching indices are $2$.
Since $S$ is an essential $2$-suborbifold, it does not bound a $3$-orbifold whose underlying space is $B^3$ and the singular locus is an unknotted arc.
It follows that the underlying space of $S$ is a $2$-sphere that splits $L$ as a non-trivial connected sum, and thus $L$ is not a small link.

We next consider the case where $\OO(L)$ is irreducible.
By the assumption, $\OO(L)$ contains an embedded closed essential $2$-suborbifold $F$.
If $F \cap L =\emptyset$, then $F \subset E(L)$ is incompressible in $E(L)$ since it is incompressible in $\OO(L)$. 
Moreover, $F$ cannot be boundary parallel in $E(L)$, otherwise $F$ would be the boundary of a regular neighborhood of a component of the singular locus of $\OO(L)$. 
In this case, a meridian curve on $F$ would bound a discal $2$-suborbifold in $\OO(L)$, contradicting the incompressibility of $F$.
If $F \cap L \neq \emptyset$, then $F' = F \cap E(L)$ is a properly embedded surface whose boundary is a collection of meridian curves.
If the surface $F'$ is compressible in $E(L)$, then there is an embedded disk $D \subset E(L)$ such that $D \cap F' = \partial D = D \cap F$ is an essential curve on $F'$.
Since $F$ is incompressible in $\OO(L)$, this curve must bound a discal $2$-suborbifold $\Delta \subset F$ with one singular point.
Therefore, $D \cup \Delta$ would be an embedded $2$-sphere in $S^3$ which meets $L$ in a single point.
This is not possible because an embedded $2$-sphere is separating in $S^3$.
Moreover, no component of $F'$ is parallel to a subsurface of $\partial E(L)$ since such a component would be compressible in $\OO(L)$.
Hence, $L$ is not small.
We have proved that being small for $L$ implies that the orbifold $\OO(L)$ is small.
To prove the converse, we need the following claim.

\begin{claim}\label{claim:annuluscompression}
If $L$ is not a small link, its exterior $E(L)$ contains a properly embedded essential orientable surface $S$ whose boundary is empty or a collection of meridian curves and which does not admit a meridional compressing annulus.
\end{claim}

Here, a \emph{meridional compressing annulus} for a properly embedded surface $S$ in $E(L)$ is an embedded annulus $A = S^1 \times [0,1]$ such that
\begin{enumerate}[label=(\roman*)]
\item $S^1 \times \{0\} = A \cap S$ is an essential, non-boundary-parallel curve on $S$ which does not bound any compressing disk for $S$;
\item $S^1 \times \{1\} = A \cap \partial E(L)$ is a meridian curve.
\end{enumerate}

Let us assume that $L$ is not a small link.
By the Claim~\ref{claim:annuluscompression}, $E(L)$ contains a properly embedded essential surface $S$ whose boundary is empty or a collection of meridian curves and which does not admit any meridional compressing annulus.
Let $F \subset \OO(L)$ be the closed embedded $2$-suborbifold obtained by gluing each boundary component of $S$ to the corresponding meridian disk with a singular point of branching index $2$.
Suppose, to the contrary, that $F$ is compressible in $\OO(L)$, that is, there is a compressing disk $\Delta$ for $F$.
It must be a discal $2$-orbifold with a singular point of branching index $2$ since $S= F \cap E(L)$ is essential in $E(L)$. 
But such a compressing disk $\Delta$ corresponds to a meridional compressing annulus $A = \Delta \cap E(L)$ for $S$, which is impossible by the choice of $S$.
Hence, the orbifold $\OO(L)$ contains an essential closed orientable $2$-suborbifold and so is not small.
\end{proof}

\begin{proof}[Proof of Claim~\ref{claim:annuluscompression}]
If $L$ is not small, $E(L)$ contains a properly embedded essential surface $S$ whose boundary is empty or a collection of meridian curves. 
A meridional annulus-compression on $S$ increases the number of boundary components while fixing the Euler characteristic of $S$.
Moreover, some component of the resulting surface must be essential in $E(L)$ by \cite[Proposition~6]{Dun88}.
Therefore, the process must stop after finitely many meridional annulus-compressions, and some component $S'$ of the resulting surface is essential and no longer admits any meridional annulus-compression.
\end{proof}

\begin{proof}[Proof of Theorem~\ref{thm:small}(1)]
Let $L$ be a small link and $L'$ a link such that $L \succeq L'$. 
Suppose, to the contrary, that the link $L'$ is not small.
By Lemma~\ref{lem:small}, the $3$-orbifold $\OO(L')$ is not small, that is, it contains an orientable closed incompressible $2$-suborbifold $F$.
Therefore, its orbifold fundamental group $G^\orb(L')$ splits along the orbifold fundamental group $\pi_{1}^\orb(F)$ as an amalgamated free product or an HNN extension. 
In particular, $G^\orb(L')$ acts non-trivially on the Bass-Serre tree $\mathcal{T}$, without edge inversions, associated to this algebraic decomposition.
The epimorphism $\varphi\colon G^\orb(L) \twoheadrightarrow G^\orb(L') $ induces a non-trivial action of the group $G^\orb(L)$ on the Bass-Serre tree $\mathcal{T}$ without edge inversions.
It follows from \cite[Corollary~10.2]{Yok16} that the good orbifold $\OO(L)$ contains an orientable incompressible $2$-suborbifold, and thus it is not small.
By Lemma~\ref{lem:small}, this contradicts the assumption that $L$ is a small link.
\end{proof}

\begin{corollary}\label{cor:knot_order}
Let $K$ and $K'$ be two knots such that $K \succeq K'$.
\begin{enumerate}[label=\textup{(\arabic*)}]
\item If $K$ is a Montesinos knot with three rational tangles, then $K'$ is the unknot, a $2$-bridge knot, or a Montesinos knot with three rational tangles.
\item If $K$ is a $(p, q)$-torus knot with $p$ and $q$ odd, then $K'$ is the unknot or a $(p', q')$-torus knot with $p'$ and $q'$ odd.
\end{enumerate}
\end{corollary}

\begin{proof} 
(1)
A Montesinos knot with three rational tangles is a small knot by \cite[Corollary~4.(a)]{Oer84}, hence $K'$ must be a small knot by Theorem~\ref{thm:small}(1). 
Since a small knot is prime, then by Theorem~\ref{thm:order}(3) and Remark~\ref{rem:rcomp}, $K'$ is the unknot, a $2$-bridge knot, or a Montesinos knot with $r' \leq 3$ rational tangles, and now $r' = 3$.

(2)
A torus knot is a small knot since the only essential properly embedded surface in its exterior is a fibered annulus. 
Hence $K'$ must be a small knot by Theorem~\ref{thm:small}(1) and then prime. 
By Theorem~\ref{thm:order}(4), $K'$ is the unknot, a $2$-bridge knot, an elliptic Montesinos knot, or a torus knot. 
Now, $K$ is a $(p, q)$-torus knot with $p$ and $q$ odd, and thus its $2$-fold branched cover $\Sigma_2(K)$ is the Brieskorn integral homology sphere $\Sigma (2, p, q)$, see \cite[Theorem~7.12]{JaNe83}. 
Hence $K'$ cannot be a $2$-bridge knot. 
If $K'$ is an elliptic Montesinos knot, then its $2$-fold branched cover $\Sigma_2(K')$ is a Seifert fibered integral homology sphere with finite fundamental group. Hence $\Sigma_2(K')$ must be the Poincar\'{e} sphere and $K'$ is the $(3, 5)$-torus knot (see \cite[Affirmation~2.5]{BoOt91}). 
Moreover, if $K'$ is a $(p', q')$-torus knot, then $p'$ and $q'$ must be odd by \cite[Theorem~7.12]{JaNe83} since $\Sigma_2(K')$ is an integral homology sphere.
\end{proof}

%%%%%%%
\subsection{Small links and determinants}\label{subsec:notsmall}
First, we prove the following proposition.

\begin{proposition}\label{prop:infinite} 
A link $L \subset S^3$ $\pi$-dominates every $2$-bridge link if and only if $\det L = 0$.
\end{proposition}

\begin{proof}
Suppose that a link $L$ $\pi$-dominates all the $2$-bridge links. 
Then the determinant of $L$ is divisible by any non-zero natural number. 
Hence, it should be zero.

Conversely, suppose $\det L = 0$.
Then, there is an epimorphism from the homology group $H_1(\Sigma_2(L); \Z)$ onto $\Z$. 
Let $\tau_\ast$ and $\tau_\sharp$ denote the automorphisms induced by the covering involution $\tau$ of 
$\Sigma_2(L)$ on $\pi_1(\Sigma_2(L))$ and $H_1(\Sigma_2(L); \Z)$, respectively.
By \cite[Theorem~1]{Fox72}, we have $\tau_\sharp = -\id$. 
Hence, defining an action of $\Z/2\Z$ on $\Z$ by $-\id$, we conclude that the epimorphisms $\pi_1(\Sigma_2(L)) \to H_1(\Sigma_2(L); \Z) \to \Z$ are $\Z/2\Z$-equivariant. Therefore, one gets a commutative diagram
\[
\xymatrix{
G^\orb(L) \cong \pi_1(\Sigma_2(L)) \rtimes_{\tau_\ast} \Z/2\Z \ar@{->>}[r] \ar@{->>}@<3em>[d]_-{\varphi_1} & \Z/2\Z \ar[d]^-{\id} \\
\Z/2\Z \ast \Z/2\Z \cong \Z \rtimes_{-\id} \Z/2\Z \ar@{->>}[r] & \Z/2\Z.}
\]
So $L$ $\pi$-dominates the split link $U \sqcup U$, which in turn $\pi$-dominates every $2$-bridge link because of the commutative diagram
\[
\xymatrix{
\Z/2\Z \ast \Z/2\Z \cong \Z \rtimes_{-\id} \Z/2\Z \ar@{->>}[r] \ar@{->>}@<2em>[d]_-{\varphi_2} & \Z/2\Z \ar[d]^-{\id} \\
D_{2n} \cong \Z/n\Z \rtimes_{-\id} \Z/2\Z \ar@{->>}[r] & \Z/2\Z.}
\]
This completes the proof.
\end{proof}

\begin{remark} 
By definition, a link $L$ $\pi$-dominates a $2$-bridge link if and only if the link group $G(L)$ admits an epimorphism onto a dihedral group which sends the meridians to reflections. 
In \cite{IIMS25}, the authors consider more general epimorphisms from link groups onto dihedral groups, in which certain meridians are mapped to reflections and the rest to rotations.
Such epimorphisms do not exist for knot groups. 
Moreover, they do not factorize through the $\pi$-orbifold group of the link, and hence do not correspond to a $\pi$-domination. 
It is proved in \cite{IIMS25} that a link with at least three components admits such an epimorphism on a dihedral group of order $2n$ for any $n \geq 3$.
The same result for a link with two components is true if and only if the linking number of the two components is even.
\end{remark}

The following property of small links will be useful for the proof of Theorem~\ref{thm:small}(2).

\begin{corollary}\label{cor:nonzerodeterminant}
A small link $L \subset S ^3$ has a non-vanishing determinant.
\end{corollary}

\begin{proof}
Let $L \subset S^3$ be a small link.
By Lemma~\ref{lem:small}, the associated $3$-orbifold $\OO(L)$ is small.
It follows from $\PSL(2, \C)$-Culler-Shalen theory in the setting of $3$-orbifolds that the $\PSL(2,\C)$-character variety $\X(\OO(L))$ is $0$-dimensional, and thus consists of finitely many points. 
See \cite{CuSh83}, \cite[Sections~ 3 and 4]{BoZh98}, and \cite[Corollary~10.2]{Yok16}.

Suppose, to the contrary, that $\det L = 0$.
By the proof of Proposition~\ref{prop:infinite}, $G^\orb(L)$ surjects onto the infinite dihedral group $\Z/2\Z \ast \Z/2\Z $. 
This induces an injection $\X(\Z/2\Z \ast \Z/2\Z)\to \X(\OO(L))$ between $\PSL(2,\C)$-character varieties.
Here, $\Z/2\Z \ast \Z/2\Z$ surjects onto all finite dihedral groups $D_{2n}$ and those can be realized as subgroups of $\PSL(2, \C)$ by \cite[C.6 and Theorem~C.9 in Chapter~V]{Mas88} (see also \cite[Lemma~4.7]{ReWa99}). 
For $n\geq 2$, the elliptic elements of order $n$ generating the normal subgroup $\Z/n\Z \subset D_{2n}$ have distinct traces.
Therefore, $\X(\Z/2\Z \ast \Z/2\Z)$ is infinite, which contradicts the fact that $\X(\OO(L))$ contains only finitely many points. 
\end{proof}

One can ask whether a link in $S^3$ could $\pi$-dominate only finitely many links with $\geq 3$ bridges, without any assumption on the determinant. 
However, it is not the case.

\begin{proposition}\label{prop:notsmall} 
There are infinitely many hyperbolic links $L^\ast$ in $S^3$ such that each of them $\pi$-dominates infinitely many links with at least three bridges, $L^\ast$ is not small, and $\det L^\ast =0$.
\end{proposition}

\begin{proof}
A split link $L = L_1 \sqcup L_2$ such that $L_2$ has $\geq 2$ bridges $\pi$-dominates infinitely many links with $\geq 3$ bridges. This follows from the fact that
\[
L_1 \sqcup L_2 \succeq U \sqcup L_2 = (U \sqcup U) \sharp L_2 \succeq L(\tfrac{p}{q}) \sharp L_2
\]
for every $2$-bridge link $L(\frac{p}{q})$.

Such a pair $(S^3, L)$ with $L = L_1 \sqcup L_2$ is a good pair in the sense of Kawauchi. 
It follows from Kawauchi's almost identical imitation theory that the pair $(S^3, L)$ admits almost identical imitations $(S^3, L^{\ast})$ which are hyperbolic links with the same number of components as $L$ (see \cite[Theorem~1.1]{Kaw89}). 
By the construction, the imitation map $q \colon (S^3, L^{\ast}) \to (S^3, L)$ has the property that the restriction $q\vert \colon (S^3, L^{\ast} \setminus{K^{\ast}}) \to (S^3, L \setminus{K})$ is homotopic to a diffeomorphism for any connected components $K^{\ast}$ of $L^{\ast}$ with $q(K^{\ast})= K$.
Hence, the imitation map induces a proper degree-one map from the exterior $E(L^{\ast})$ onto $E(L)$ which sends the meridians of $L^{\ast}$ to those of $L$. 
This degree-one map is $\pi_1$-surjective, and thus induces the commutative diagram
\[
\xymatrix{
G^\orb(L^\ast) \ar@{->>}[r] \ar@{->>}[d] & \Z/2\Z \ar[d]^-{\id} \\
G^\orb(L) \ar@{->>}[r] & \Z/2\Z.}
\]
It follows that $L^{\ast} \succeq L =  L_1 \sqcup L_2 \succeq  L(\frac{p}{q}) \sharp L_2$ for every $2$-bridge link $L(\frac{p}{q})$. 

Since $L^{\ast}$ $\pi$-dominates non-small knots by construction, Theorem~\ref{thm:small}(1) shows that $L^\ast$ cannot be small. 
Moreover, it follows from $\det L = 0$ that $\det L^\ast = 0$.
\end{proof}

The following question remains open.

\begin{question}
Is it possible for a link $L \subset S^3$ to $\pi$-dominate infinitely many prime links with at least three bridges?
\end{question}

%%%%%%%%%%%%
\subsection{Proof of Theorem~\ref{thm:small}(2)}
\label{subsec:finiteness}

We prove now that a small link $L$ $\pi$-dominates only finitely many distinct links.
The proof uses the $\PSL(2,\C)$-character variety as in the article \cite{ReWa99} for the study of non-zero degree maps between non-Haken $3$-manifolds. 
We need to work in the setting of orbifolds since the $2$-fold branched cover of a small link may be a Haken manifold.
In the following, we write $\X(\OO(L)) = \X(G^\orb(L))$ for the $\PSL(2,\C)$-character variety of the orbifold $\OO(L)$ and $\X^\irr(\OO(L))$ for the subspace of irreducible characters (see \cite[Section~ 3]{BoZh98}).

Let $L \subset S^3$ be a small link.
By Lemma~\ref{lem:small}, the associated $3$-orbifold $\OO(L)$ is small.
It follows from $\PSL(2, \C)$-Culler-Shalen theory \cite[Section~ 4]{BoZh98} and \cite[Corollary~10.2]{Yok16} that the $\PSL(2,\C)$-character variety $\X(\OO(L))$ is $0$-dimensional, and thus consists of finitely many points.
By Theorem~\ref{thm:small}(1), any link $L'$ such that $L \succeq L'$ is small and the orbifold $\OO(L')$ is small by Lemma~\ref{lem:small}.
It follows that the orbifold $\OO(L')$ is irreducible and atoroidal.
Hence, by the geometrization theorem of $3$-orbifolds (see \cite[Theorem~2]{BoPo01}), $\OO(L)$ is hyperbolic or Seifert fibered with a small base which is either hyperbolic, Euclidean, elliptic or bad.
We distinguish three mutually exclusive cases according to the geometry of $\OO(L')$ and prove the finiteness result for each case.

We recall that when the orbifold $\OO(L')$ is hyperbolic, the link $L'$ is said to be $\pi$-hyperbolic.

\begin{claim}\label{claim:hyperbolic} 
A small link $L$ $\pi$-dominates at most $\vert \mathcal{X}^{\irr}(\mathcal{O}(L)) \vert$ $\pi$-hyperbolic links. 
\end{claim}

\begin{proof}[Proof of Claim~\ref{claim:hyperbolic}]
We associate to each $\pi$-hyperbolic link $L'$ a faithful discrete representation $\rho'\colon \pi_1^\orb(\OO(L')) \to \PSL(2,\C)$ corresponding to the holonomy of the hyperbolic structure on $\OO(L')$.
The epimorphism $\varphi \colon G^\orb(L) \twoheadrightarrow G^\orb(L')$ induces an injective map $\varphi^*\colon \X(\OO(L'))\to \X(\OO(L))$ defined by $\varphi^{*} (\chi_{\rho'}) = \chi_{\rho'\circ \varphi}$.
Therefore, each $\pi$-hyperbolic link $L'$ $\pi$-dominated by $L$ gives an irreducible character $\varphi^{*} (\chi_{\rho'})\in \X^\irr(\OO(L))$, which determines the irreducible representation $\rho'\circ \varphi \colon G^\orb(L) \to \PSL(2,\C)$ up to conjugation in $\PSL(2,\C)$, and thus its image $\rho'(G^\orb(L'))$ up to isomorphism.
See \cite[Section~ 3]{BoZh98}.
Since the representation $\rho'$ is  faithful, the irreducible character $\varphi^{*} (\chi_{\rho'})\in \X^\irr(\OO(L))$ determines the $\pi$-orbifold group $G^\orb(L')$ up to isomorphism.
By \cite[Theorem~1]{BoZi89}, the $\pi$-orbifold group $G^\orb(L')$ determines the pair $(S^3, L')$ up to homeomorphism, hence the irreducible character $\varphi^{*} (\chi_{\rho'})\in \X^\irr(\OO(L))$  determines a unique $\pi$-hyperbolic link $L'$ up to equivalence.
Therefore, the number of distinct $\pi$-hyperbolic links $L'$ with $L \succeq L'$ is smaller than or equal to the number of irreducible characters in $\X^\irr(\OO(L))$.
\end{proof}

We consider now the case where $\mathcal{O}(L')$ is a small Seifert fibered orbifold. 
Then $L'$ is a small Seifert link or a small generalized Montesinos link. 
A generalized Montesinos link is the union of a classical Montesinos link with $g$ unknotted parallel components which surround the central band as in \cite[Figure~7]{BoZi85}. 
For $g \geq 1$, its $2$-fold branched cover is Seifert fibered with a non-orientable base of genus $g$.
See \cite{BoSi85}, \cite{Mon73}, \cite[Sections~4.7--4.8]{Mon87} for details.

When $L'$ is a small Seifert link, the base of $\mathcal{O}(L')$ is a $2$-dimensional orientable orbifold with underlying space $S^2$ and at most three conical singular points with cyclic local isotropy group $\Z/\alpha \Z$ with $\alpha > 1$. 
Then $ L'$ has at most three components and its $2$-fold branched cover is Seifert fibered with an orientable base $S^2$ and at most three exceptional fibers. 
If $L'$ is a small generalized Montesinos link, its $2$-fold branched cover is Seifert fibered with an orientable base $S^2$ and at most three exceptional fibers or with a non-orientable base of genus $1$ and at most one exceptional fiber (see \cite{Mon73}). 
In the last case, the $2$-fold branched cover is an elliptic Seifert fibered manifold which also carries a Seifert fibration with base $S^2$ and three exceptional fibers. 
In this case, it has been shown in \cite{Mon73} that the generalized Montesinos link is equivalent to a classical Montesinos link with three rational tangles (see also \cite[Section~2.2]{BoZi85}). 
Therefore, in all cases we can assume that $L'$ is a Seifert fibered link with at most three components or a Montesinos link with at most three rational tangles, and the $2$-fold branched cover $\Sigma_2(L')$ is a Seifert fibered $3$-manifold with base $S^2$ and at most three singular points. 
To prove the finiteness result when $\mathcal{O}(L')$ is Seifert fibered, the following result will be useful.

\begin{lemma}\label{lem:seifertfiniteness} 
Let $L\subset S^3$ be a Seifert or a Montesinos link. 
If the orbifold base $\mathcal{B}$ of the Seifert fibered $3$-orbifold $\OO(L)$ can take only finitely many topological types and $\det L$ is bounded, then $L$ can take only finitely many link types.
\end{lemma}

\begin{proof}
The classification of Seifert fibered closed oriented $3$-manifolds with orientable base shows that $\Sigma_2(L)$ is determined, up to orientation-preserving homeomorphism, by the topological type of its orbifold base $\mathcal{B}$, the local monodromy fractions $\frac{\beta_i}{\alpha_i} \in \Q/\Z$ associated to each exceptional fiber and the rational Euler number $e_0(\Sigma_2(L)) \in \Q$. 
The topological type of the orbifold base $\mathcal{B}$ of $\OO(L)$ determines finitely many possible orbifold bases $\overline{\mathcal{B}}$ for $\Sigma_2(L)$ since $\overline{\mathcal{B}}$ is either homeomorphic to $\mathcal{B}$ or a $2$-fold cover of $\mathcal{B}$. 
The orbifold base $\overline{\mathcal{B}}$ determines the order $\alpha_i > 1$ of each exceptional fiber of the Seifert manifold $\Sigma_2(L)$. 
Then, for each $\alpha_i$, there are at most $\alpha_i - 1$ integers $0 < \beta_i < \alpha_i$ relatively prime to $\alpha_i$, and hence there are at most $\alpha_i - 1$ distinct possible monodromy fractions $\frac{\beta_i}{\alpha_i} \in \Q/\Z$ associated to each exceptional fibers. 

Moreover, since $\det L = \vert H_1(\Sigma_2(L); \Z) \vert$ can take only finitely many values, the rational Euler number $e_0(\Sigma_2(L)) = \frac{\vert H_1(\Sigma_2(L); \Z) \vert}{\prod_i \alpha_i}$ takes only finitely many values. 
Therefore, each orbifold base $\overline{\mathcal{B}}$ determines at most finitely many Seifert fibered manifolds $\Sigma_2(L)$. 
It follows that given finitely many possible topological types for the orbifold base $\mathcal{B}$ of $\OO(L)$ and the upper bound on $\det L$, the $2$-fold branched cover $\Sigma_2(L)$ can take only finitely many topological types. 
Now, the lemma follows from the fact that there are only finitely many distinct Seifert fibered links and Montesinos links with the same $2$-fold branched cover since there are at most finitely many conjugacy classes of fiber-preserving involutions on a given closed orientable Seifert fibered $3$-manifold with orientable base.
\end{proof}

We first consider the case where the base of the Seifert fibered orbifold $\mathcal{O}(L')$ is hyperbolic.

\begin{claim}\label{claim:hyperbolicbase} 
A small link $L$ $\pi$-dominates at most finitely many links $L'$ such that $\mathcal{O}(L')$ is Seifert fibered with a hyperbolic base.
\end{claim}

\begin{proof}[Proof of Claim~\ref{claim:hyperbolicbase}]
Since $L$ is small, its determinant $\det L$ does not vanish by Corollary~\ref{cor:nonzerodeterminant}. 
Let $L'$ such that $L \succeq L'$ and that the small orbifold $\mathcal{O}(L')$ is Seifert fibered with a hyperbolic base 
$\mathcal{B}'$. 
Using the quotient epimorphism $\pi_1^\orb(\OO(L'))  \twoheadrightarrow  \pi_1(\mathcal{B}')$, we associate to each such link $L'$  an irreducible representation
\[
\rho'\colon \pi_1^\orb(\OO(L'))  \twoheadrightarrow  \pi_1(\mathcal{B}') \to \PSL(2,\R) \subset \PSL(2,\C)
\]
given by the faithful discrete representation $\pi' \colon \pi_1(\mathcal{B}') \to \PSL(2,\R)$ corresponding to the holonomy of a hyperbolic structure on the $2$-orbifold base $\mathcal{B}'$.
Like in the hyperbolic case above, each link $L'$ $\pi$-dominated by $L$ gives an irreducible character $\varphi^{*} (\chi_{\rho'})\in \X^\irr(\OO(L))$ which determines the irreducible representation $\rho'\circ \varphi \colon G^\orb(L) \to \PSL(2,\C)$ up to conjugation in $\PSL(2,\C)$ and thus its image $\rho'(G^\orb(L')) = \pi'(\pi_1(\mathcal{B}'))$ up to isomorphism.
Since the representation $\pi'$ is faithful, the irreducible character $\varphi^{*} (\chi_{\rho'})\in \X^\irr(\OO(L))$  determines the orbifold group $\pi_1(\mathcal{B}')$ up to isomorphism.
A hyperbolic closed $2$-dimensional orbifold is determined up to homeomorphism by its orbifold fundamental group (see \cite{Mac65}), therefore the irreducible character $\varphi^{*} (\chi_{\rho'})\in \X^\irr(\OO(L))$ determines the base $\mathcal{B}'$ of the Seifert fibered orbifold $\OO(L')$ up to homeomorphism. 
Since there are at most finitely many irreducible characters in $\X^\irr(\OO(L))$, there are at most finitely many possible topological types for the orbifold base $\mathcal{B}'$ of the Seifert fibered orbifold $\OO(L')$. 
There is an epimorphism $\tilde{\varphi}\colon \pi_1(\Sigma_2(K)) \twoheadrightarrow \pi_1(\Sigma_2(L'))$ and thus the order $\vert H_1(\Sigma_2(L'); \Z) \vert$ of the first homology group of $\Sigma_2(L')$ divides $\vert H_1(\Sigma_2(L); \Z) \vert$.
Therefore, the value $\det L'= \vert H_1(\Sigma_2(L'); \Z) \vert$ is bounded and the finiteness of the link types for $L'$ follows from Lemma~\ref{lem:seifertfiniteness}.
\end{proof}

When $\mathcal{O}(L')$ is Seifert fibered with an elliptic, Euclidean, or bad base, we have a more general result.

\begin{claim}\label{claim:ellipticeuclidean} 
Any link $L$ with $\det L \neq 0$ $\pi$-dominates at most finitely many links $L'$ such that $\mathcal{O}(L')$ is Seifert fibered with an elliptic, Euclidean, or bad orbifold base.
\end{claim}

\begin{proof}[Proof of Claim~\ref{claim:ellipticeuclidean}] 
If the orbifold base of $\OO(L')$ is bad or spherical with two singular points, then the base of the $2$-fold branched cover $\Sigma_2(L')$ is a $2$-sphere with at most two singular points. 
Hence, it is a Lens space and $L'$ is a $2$-bridge link of type $\frac{p}{q}$ with $1 \leq q < p$ and $\gcd(p,q) = 1$ such that $p= \det L'$ divides $\det L \neq 0$. 
Therefore, $p$ takes only finitely many values. 
Moreover, there are at most $p - 1$ possible values for $q$, and thus at most finitely many possible distinct $2$-bridge links $L'$ of type $\frac{p}{q}$ with $1 \leq q < p$ and $\gcd(p,q) = 1$.

We suppose now that the orbifold base of $\OO(L')$ is elliptic with at least three singular points or Euclidean. 
Then the orbifold base $\overline{\mathcal{B}'}$ of $\Sigma_2(L')$ is elliptic or Euclidean and takes only finitely many topological types, except for the elliptic base $S^2(2, 2, n)$ with underlying space $S^2$ and three singular points of order $\{2, 2, n\}$ or $P^2(n)$ with underlying space the projective plane $P^2$ and a single singular point of order $n$. 
Therefore, if $\overline{\mathcal{B}'} \not \in \{ S^2(2, 2, n), P^2(n)\}$, the proof of Lemma~\ref{lem:seifertfiniteness} shows that there are finitely many possible link types for $L'$.

When $\overline{\mathcal{B}'} = S^2(2, 2,n)$, $\Sigma_2(L')$ is a prism manifold and the order $n$ of the exceptional fiber is bounded above by the order of the first homology group $\vert H_1(\Sigma_2(L'); \Z) \vert = \det L'$ which divides $\det L \neq 0$, see \cite[Theorem~2(ii), Section~6.2]{Orl72}. 
Therefore, the orbifold base of $\Sigma_2(L')$ can take only finitely many topological types, and the proof of Lemma~\ref{lem:seifertfiniteness} shows that $L'$ can take only finitely many link types in this case too. 
If $\overline{\mathcal{B}'} = P^2(n)$, $\Sigma_2(L')$ admits also a Seifert fibration with base $S^2(2, 2,n)$ and thus is a prism manifold, see \cite[Theorem~2(vi), Section~6.2]{Orl72}.
The finiteness result for $L'$ follows from the previous case.
\end{proof}

%%%%%%%%%%%%%%%%%%%%%%%%%%
\section{$\pi$-minimal links}\label{sec:piminimal}

\begin{definition}
A link $L$ is said to be \emph{$\pi$-minimal} if $L \succeq L'$ implies that $G^{\orb}(L')$ is isomorphic to either $\Z/2\Z$, a dihedral group, or $G^{\orb}(L)$.
\end{definition}

By definition, the unknot and $2$-bridge links are $\pi$-minimal.
In contrast, Proposition~\ref{prop:notsmall} gives examples of non-$\pi$-minimal links $L$ with $\det L=0$.
The next proposition will be proved by Lemma~\ref{lem:minimalnotprime} and Proposition~\ref{prop:sequence}.

\begin{proposition} \label{prop:minimal}
Any link with at least three bridges $\pi$-dominates a $\pi$-minimal link with at least three bridges which is either prime or a connected sum of two $2$-bridge links.
\end{proposition}

First, we determine the $\pi$-minimal links which are not prime.

\begin{lemma} \label{lem:minimalnotprime}
If a $\pi$-minimal link is not prime, then it is a connected sum of two $2$-bridge links or a $2$-component trivial link.
\end{lemma}

\begin{proof}
A link that is not prime is either split or a non-trivial connected sum. 
By Proposition~\ref{prop:firstproperties}(4), a split link $L = L_1 \sqcup L_2$ $\pi$-dominates a connected sum $L_1 \sharp L_2$, and thus it is not minimal unless $L_1 = U$ is the unknot and $L_2 $ is a $2$-bridge link. 
Therefore, $L = U \sqcup L_2 = (U \sqcup U) \sharp L_2$. 
By the proof of Proposition~\ref{prop:infinite}, the split link $U \sqcup U$ $\pi$-dominates every $2$-bridge link, so $L = U \sqcup L_2$ $\pi$-dominates the connected sum of $L_2$ with every $2$-bridge link. 
It follows that $L = U \sqcup L_2$ cannot be $\pi$-minimal unless $L_2$ is the unknot and $L = U \sqcup U$ is the $2$-component trivial link.

%Therefore, 
%\[
%G^{\orb}(L_1 \sqcup L_2) = G^{\orb}(U \sqcup L(\tfrac{p}{q})) = \Z/2\Z \ast D_{2p} = (\Z/2\Z \ast \Z/2\Z) \ast_{\Z/2\Z} D_{2p}.
%\]
%The free product $\Z/2\Z \ast \Z/2\Z$ surjects onto any dihedral group $D_{2q}$, therefore the group $G^{\orb}(L_1 \sqcup L_2)$ surjects onto the group $D_{2q} \ast_{\Z/2\Z} D_{2p}$, for any $q \geq 1$.  
%It follows that $L_1 \sqcup L_2$ cannot be $\pi$-minimal unless $L_2$ is the unknot and $L_1 \sqcup L_2$ is the $2$-component trivial link.

Let $L= L_1 \sharp L_2$ with $L_1$ and $L_2$ not being the unknot. 
If $L$ is $\pi$-minimal, $L_1$ and $L_2$ must be $2$-bridge links since $L$ $\pi$-dominates $L_1$ and $L_2$ by Proposition~\ref{prop:firstproperties}(3), and $G^{\orb}(L)$ is not isomorphic to $G^{\orb}(L_1)$ or to  $G^{\orb}(L_1)$. 
This follows from the fact that $\pi_1(\Sigma_2(L)) = \pi_1(\Sigma_2(L_1))\ast \pi_1(\Sigma_2(L_2))$ is not isomorphic to $\pi_1(\Sigma_2(L_1))$ or $\pi_1(\Sigma_2(L_2))$ when none of these groups is trivial. 
\end{proof}

Next, we show the existence of $\pi$-minimal links with at least three bridges.

\begin{proposition}\label{prop:sequence}
Let $L_1 \succeq L_2 \succeq \cdots \succeq L_i \succeq L_{i+1} \succeq \cdots ,$ be a $\pi$-ordered sequence of links with at least three bridges in $S^3$. 
Then there exists some integer $n_0$ such that $G^{\orb}(L_{n_0}) \cong G^{\orb}(L_{n})$ for $n \geq n_0$.
\end{proposition} 

\begin{proof}
By Lemma~\ref{lem:epicover}, the sequence of epimorphisms
\[
G^{\orb}(L_1) \xrightarrow{\varphi_1}  G^{\orb}(L_2) \xrightarrow{\varphi_2 }\cdots \xrightarrow{\varphi_{i-1}}  G^{\orb}(L_i) \xrightarrow{\varphi_i} \cdots
\]
induces a sequence of epimorphisms 
\[
\pi_1(\Sigma_2(L_1)) \xrightarrow{\tilde{\varphi}_1}\pi_1(\Sigma_2(L_2)) \xrightarrow{\tilde{\varphi}_2} \cdots \xrightarrow{\tilde{\varphi}_{i-1}} \pi_1(\Sigma_2(L_i))  \xrightarrow{\tilde{\varphi}_i} \cdots.
\]
By a deep result of Groves, Hull, and Liang~\cite[Theorem~B]{GHL25}, this infinite sequence of epimorphisms between the fundamental groups of compact $3$-manifolds eventually contains an isomorphism. 
Therefore, the sequence of epimorphisms between orbifold groups stabilizes.
\end{proof}

\begin{proof}[Proof of Proposition~\ref{prop:minimal}]
By Proposition~\ref{prop:sequence}, any link $L$ with at least three bridges $\pi$-dominates a $\pi$-minimal link $L'$ with at least three bridges. 
Otherwise starting with $L = L_1$ one could build a $\pi$-ordered  infinite sequence $L_1 \succeq  \cdots \succeq L_n \succeq \cdots$
of links with at least three bridges such that $G^{\orb}(L_{n}) \not \cong G^{\orb}(L_{m})$ for $n \neq m$. 
If the $\pi$-minimal link $L'$ is not prime, then it is a connected sum of two $2$-bridge links or a $2$-component trivial link by Lemma~\ref{lem:minimalnotprime}. 
The link $L'$ cannot be a $2$-component trivial link since it has at least three bridges, and hence it must be a connected sum of two $2$-bridge links.
\end{proof}

The following criterion of $\pi$-minimality can be useful.

\begin{lemma}\label{lem:minimal2fold} 
Let $L \subset S^3$ be a link with at least three bridges.
If $\pi_1(\Sigma_2(L))$ does not surject onto the fundamental group of a closed orientable $3$-manifold which is neither trivial, cyclic, nor isomorphic to $\pi_1(\Sigma_2(L))$, then $L$ is $\pi$-minimal.
\end{lemma}

\begin{proof}
Let $L \subset S^3$ be a link with at least three bridges that is not $\pi$-minimal. 
Then $L \succeq L'$, where $L'$ has at least three bridges and there is an epimorphism $\varphi\colon G^\orb(L) \twoheadrightarrow G^\orb(L')$ which is not an isomorphism. 
By Lemma~\ref{lem:epicover}, this epimorphism induces an epimorphism $\tilde{\varphi}\colon \pi_1(\Sigma_2(L)) \twoheadrightarrow \pi_1(\Sigma_2(L'))$. 
Since $L'$ has at least three bridges, $\Sigma_2(L')$ cannot be cyclic or trivial, otherwise $L'$ would be a $2$-bridge link or the unknot by the orbifold theorem~\cite[Theorem~1]{BoPo01}. 
Suppose, to the contrary, that $\pi_1(\Sigma_2(L'))$ is isomorphic to $\pi_1(\Sigma_2(L))$.
Then, $\tilde{\varphi}$ is an isomorphism because $3$-manifold groups are hopfian by \cite{Hem87}. 
Hence, the commutative diagram below (see also the proof of Lemma~\ref{lem:epicover}) shows that $\varphi$ itself is an isomorphism, which is not possible:
\[
\xymatrix{
1\ar[r] & \pi_1(\Sigma_2(L)) \ar[r] \ar@{->>}[d]^-{\tilde{\varphi}}& G^\orb(L) \ar[r] \ar@{->>}[d]^-{\varphi} & \Z/2\Z \ar[d]^-{\id} \ar[r] & 1 \\
1\ar[r] & \pi_1(\Sigma_2(L')) \ar[r] & G^\orb(L') \ar[r] & \Z/2\Z \ar[r] & 1.}
\]
Therefore, the group $\pi_1(\Sigma_2(L))$ surjects onto the fundamental group of a closed $3$-manifold which is neither trivial, cyclic, nor isomorphic to $\pi_1(\Sigma_2(L))$.
\end{proof}

As an application, we give examples of infinite families of $\pi$-minimal knots.

\begin{proposition}\label{prop:minimal_example}
The following hold.
\begin{enumerate}[label=\textup{(\arabic*)}]
\item A $(p, q)$-torus knot with $p$ and $q$ prime is $\pi$-minimal.
\item Let $K = K(\frac{\beta_1}{\alpha_1}, \frac{\beta_2}{\alpha_2}, \frac{\beta_3}{\alpha_3})$ be a non-elliptic Montesinos knot. If $2 \leq \alpha_1 < \alpha_2 < \alpha_3$ are prime numbers, then the Montesinos knot $K$ is $\pi$-minimal.
\end{enumerate}
\end{proposition}

\begin{proof}
Let $K$ be a $(p, q)$-torus knot with $2 \leq p < q$ and $p$, $q$ coprime. 
Let us assume that $p$ and $q$ are prime. 
If $p$ or $q$ is even, then $p = 2$ and the torus knot $T(2, q)$ is $\pi$-minimal. 
So, in case (1), we can assume that $p$ and $q$ are odd primes. 
Then it follows that in both cases (1) and (2) the $2$-fold branched cover $\Sigma_2(K)$ is Seifert fibered with orbifold base $S^2(\alpha_1, \alpha_2, \alpha_3)$, where $2 \leq \alpha_1 < \alpha_2 < \alpha_3$ are prime numbers. 
One has $(\alpha_1, \alpha_2, \alpha_3) = (2, p, q)$ in case (1). 
In case (2), since $K$ is assumed not to be an elliptic Montesinos knot, the hypotheses imply that $\frac{1}{\alpha_1} + \frac{1}{\alpha_2} + \frac{1}{\alpha_3} < 1$ unless $(\alpha_1, \alpha_2, \alpha_3) = (2, 3, 5)$ and $K$ is the $(3, 5)$-torus knot.
Then the proof of Proposition~\ref{prop:minimal_example} follows from Lemma~\ref{lem:minimal2fold} and the next claim.
\end{proof}

\begin{claim} \label{claim:isomorphism}
Let $K$ be a $(p,q)$-torus knot or a Montesinos knot as in Proposition~\ref{prop:minimal_example}.
If there is an epimorphism $\varphi \colon \pi_1(\Sigma_2(K)) \twoheadrightarrow \pi_1(M)$ where $M$ is a closed orientable $3$-manifold and $\pi_1(M)$ is not cyclic, then $\varphi$ is an isomorphism.
\end{claim}

\begin{proof} 
If the manifold $M$ is not prime, its fundamental group admits a non-trivial decomposition as a free product. 
Then the group $\pi_1(M)$ acts non-trivially, without edge inversions, on the Bass-Serre tree $\mathcal{T}$ associated to this algebraic decomposition.
The epimorphism $\varphi \colon \pi_1(\Sigma_2(K)) \twoheadrightarrow \pi_1(M)$ induces a non-trivial action, without edge inversions, of the group $\pi_1(\Sigma_2(K))$ on the Bass-Serre tree $\mathcal{T}$.
It follows from \cite{CuSh83} that the manifold $\Sigma_2(K)$ splits along some closed orientable incompressible surface. 
By Waldhausen~\cite{Wal67}, a closed incompressible surface in $\Sigma_2(K)$ is either a vertical torus which is a union of fibers or a horizontal surface transverse to the fibers of the Seifert fibration of $\Sigma_2(K)$.
In the first case, the projection of the incompressible vertical torus on the base $S^2(\alpha_1, \alpha_2, \alpha_3)$ would be an essential simple closed curve, which does not exist on such an orbifold. 
In the second case, since the base of the Seifert fibration is orientable, the horizontal surface would be non-separating in the rational homology sphere $\Sigma_2(K)$, which is impossible.
Therefore, $\Sigma_2(K)$ cannot split along some closed incompressible surface, and $M$ must be prime. 
The only prime manifold which is not irreducible is $S^1 \times S^2$. 
Since $M$ is a $\Z/2\Z$-homology sphere because $\varphi$ induces an epimorphism from the finite group of odd order $H_1(\Sigma_2(K); \Z)$ onto $H_1(M; \Z)$, $M$ must be irreducible.

If $\pi_1(M)$ is finite, by the proof of the geometrization conjecture (see \cite{KlLo08}, \cite{MoTi07}), $M$ carries a Riemannian metric with constant sectional curvature and hence  is either a lens space or an elliptic Seifert fibered $3$-manifold with three exceptional fibers and base $S^2(\alpha'_1, \alpha'_2, \alpha'_3)$ with $\frac{1}{\alpha'_1} + \frac{1}{\alpha'_2} + \frac{1}{\alpha'_3} > 1$. 
Then the only possibilities for the triple $\{\alpha'_1, \alpha'_2, \alpha'_3\}$ are $\{2, 3, 3\}, \{2, 3, 5\}$ since at most one of the $\alpha'_i$ can be even because $M$ is a $\Z/2\Z$-homology sphere.

If $\pi_1(M)$ is infinite, then $M$ is aspherical and $\pi_1(M)$ is torsion-free. 
Moreover, $\pi_1(\Sigma_2(K))$ is infinite and contains a center $Z$ which is infinite cyclic generated by a regular fiber of the Seifert fibration of $\Sigma_2(K)$. 
Since the quotient $\pi_1(\Sigma_2(K))/Z = \pi_{1}^\orb(S^2(\alpha_1, \alpha_2, \alpha_3))$ is generated by torsion elements, the image $\varphi(Z)$ cannot be trivial, and thus $\pi_1(M)$ has a non-trivial center. 
By \cite{Gab92} and \cite{CaJu94}, $M$ is a Seifert $3$-manifold and its base is orientable because the subgroup generated by the fiber is central in $\pi_1(M)$. 
Since $\pi_1(M)$ is not cyclic, it is generated by exactly two elements because $\pi_1(\Sigma_2(K))$ is generated by two elements. 
Hence the Seifert fibered $\Z/2\Z$-homology sphere $M$ admits three exceptional fibers by \cite[Theorem~1.1(ii)]{BoZi84} and a base with underlying space $S^2$ and three singular points with branching indices $\{\alpha'_1, \alpha'_2, \alpha'_3\}$ with $\frac{1}{\alpha'_1} + \frac{1}{\alpha'_2} + \frac{1}{\alpha'_3} \leq 1$.

In any case, $M$ is a Seifert fibered $3$-manifold with three exceptional fibers and base $S^2(\alpha'_1, \alpha'_2, \alpha'_3)$ with at most one of the $\alpha'_i$ even. 
Since the triangle group $T(\alpha'_1, \alpha'_2, \alpha'_3) = \pi_1^\orb(S^2(\alpha'_1, \alpha'_2, \alpha'_3))$ has no center, the epimorphism $\varphi \colon \pi_1(\Sigma_2(K)) \twoheadrightarrow \pi_1(M)$ induces an epimorphism 
\[
\bar{\varphi}\colon \pi_1^\orb(S^2(\alpha_1, \alpha_2, \alpha_3)) = T(\alpha_1, \alpha_2, \alpha_3) \twoheadrightarrow  \pi_1^\orb(S^2(\alpha'_1, \alpha'_2, \alpha'_3)) = T(\alpha'_1, \alpha'_2, \alpha'_3)
\]
between the fundamental groups of the bases.
The presentation of the triangle group 
\[
T(\alpha_1, \alpha_2, \alpha_3)=  \langle x, y, z \mid x^{\alpha_1} = y^{\alpha_2} = z^{\alpha_3} =xyz = 1 \rangle,
\]
with $2 \leq \alpha_1 < \alpha_2 < \alpha_3$ prime numbers shows that each image $\bar{\varphi}(x)$, $\bar{\varphi}(y)$ and $\bar{\varphi}(z)$ is not trivial, otherwise the image $\bar{\varphi}( T(\alpha_1, \alpha_2, \alpha_3))$ would be trivial. 
Hence, $\bar{\varphi}( T(\alpha_1, \alpha_2, \alpha_3)) = T(\alpha'_1, \alpha'_2, \alpha'_3)$ contains elements of distinct prime orders $\alpha_1, \alpha_2, \alpha_3$. 
We distinguish two cases according to whether $\frac{1}{\alpha'_1} + \frac{1}{\alpha'_2} + \frac{1}{\alpha'_3} \geq 1$ or not.

If $\frac{1}{\alpha'_1} + \frac{1}{\alpha'_2} + \frac{1}{\alpha'_3} \geq 1$, then $(\alpha'_1, \alpha'_2, \alpha'_3) \in \{(2, 3, 3), (2, 3, 5), (3, 3, 3)\}$ since at most one of the $\alpha'_i$ is even.
The triple $\{2, 3, 3\}$ is not possible since the order $12$ of $T(2, 3, 3)$ is not divisible by three distinct prime numbers $\alpha_1, \alpha_2, \alpha_3$. 
The triple $\{3, 3, 3\}$ is also not possible since the torsion elements in the Euclidean triangle group $T(3, 3, 3)$ have all the same order $3$. 
If $\{\alpha'_1, \alpha'_2, \alpha'_3\} = \{2, 3, 5\}$, then the only prime divisors of the order $60$ of $T(2, 3, 5)$ are $2, 3$ and $5$. 
Therefore, $\{\alpha_1, \alpha_2, \alpha_3\} = \{2, 3, 5\}$ and $K$ is an elliptic Montesinos knot, which is not the case in (2). 
In the case (1), $K$ is the $(3, 5)$-torus knot  and $\Sigma_2(K)$ is the Poincar\'e homology sphere whose fundamental group is the binary icosahedral group $I^{\ast}$ of order $120$. 
Then $M$ is a Seifert fibered integral homology sphere with finite fundamental group, and hence it is the Poincar\'e homology sphere and $\pi_1(M)$ is isomorphic to $\pi_1(\Sigma_2(K))$. 
In particular, the epimorphism $\varphi$ is injective because the groups have the same order.

If $\frac{1}{\alpha'_1} + \frac{1}{\alpha'_2} + \frac{1}{\alpha'_3} < 1$, then $T(\alpha'_1, \alpha'_2, \alpha'_3) \subset \PSL(2,\mathbb{R})$ is a hyperbolic triangle group.
The two elliptic elements $a = \bar{\varphi}(y)$ and $b = \bar{\varphi}(z)$ of prime orders $\alpha_2$ and $\alpha_3$ generate the discrete non-abelian group $T(\alpha'_1, \alpha'_2, \alpha'_3) \subset \PSL(2, \R)$. 
Then up to taking suitable powers $u=a^{k}$ and $v = b^{\ell}$ to normalize the matrix representatives of $u$ and $v$ in $SL(2,\R)$, at least one of Cases~(I)--(VII) in \cite[Theorem~2.3]{Kna68} holds. 
Since the triangle group $T(\alpha'_1, \alpha'_2, \alpha'_3)$ is co-compact, Case~(II) is impossible by \cite[Figures~2 and 3]{Kna68}.
Cases~(III) and (VI) do not hold because the generators $u$ and $v$ do not have the same order since $\alpha_2 < \alpha_3$.
Case~(IV) is not possible because neither of the generators $u$ nor $v$ has order $2$, since $2 \leq \alpha_1 < \alpha_2 < \alpha_3$.
If Cases~(V) or (VII) happen, then the order of $u$ is $3$ and the order of $v$ is $n \geq 7$ and not divisible by $3$. 
Hence $\alpha_2= 3$ and $\alpha_3= n$. 
It follows that $\alpha_1 = 2$ and so $\{\alpha_1, \alpha_2, \alpha_3\} = \{2, 3, n\}$.
Moreover, by \cite[Section 3, Cases~V and VII]{Kna68}, we have 
\[
\bar{\varphi}( T(\alpha_1, \alpha_2, \alpha_3)) = T(\alpha'_1, \alpha'_2, \alpha'_3) = T(2, 3, n) = T(\alpha_1, \alpha_2, \alpha_3).
\]
By the hopfian property of triangle groups, the epimorphism $\bar{\varphi}$ is an isomorphism. 
Therefore, the epimorphism $\varphi \colon \pi_1(\Sigma_2(K)) \twoheadrightarrow \pi_1(M)$ induces injective homomorphisms both on the center of the Seifert fibered manifolds and on the fundamental groups of their bases. 
Hence it induces an injective homomorphism on $\pi_1(\Sigma_2(K))$. 
Therefore, the epimorphism $\varphi$ is an isomorphism.

We are left with Case~(I) where $uv$ has an extreme negative trace. 
By \cite[Proposition~2.2]{Kna68}, $u$ and $v$ generate a hyperbolic triangle group $T(\alpha_2, \alpha_3, n)$. 
Therefore, $T(\alpha'_1, \alpha'_2, \alpha'_3) = T(\alpha_2, \alpha_3, n)$. 
The epimorphism $\bar{\varphi} \colon  T(\alpha_1, \alpha_2, \alpha_3)\twoheadrightarrow T(\alpha_2, \alpha_3, n)$ implies that $-\chi(S^2(\alpha_1, \alpha_2, \alpha_3)) \geq -\chi(S^2(\alpha_2, \alpha_3, n))$ by \cite[Lemma~2.5]{Ron92}. 
So, $1 - (\frac{1}{\alpha_1} + \frac{1}{\alpha_2} + \frac{1}{\alpha_3}) \geq 1 - (\frac{1}{\alpha_2} + \frac{1}{\alpha_3} + \frac{1}{n})$ and hence $n \leq \alpha_1$.
On the other hand, $\bar{\varphi}(x)$ is an element of prime order $\alpha_1$ in $T(\alpha_2, \alpha_3, n)$. 
The only torsion elements in a hyperbolic triangle group are conjugate in a maximal finite cyclic subgroup generated by one of the rotation generators. 
Therefore, $\alpha_1$ must divide one of the numbers $\{\alpha_2, \alpha_3, n\}$. 
Since $\alpha_2$ and $\alpha_3$ are prime numbers distinct from $\alpha_1$, the only possibility is that $\alpha_1$ divides $n$. 
Then $\alpha_1 = n$ because $n \leq \alpha_1$.
Therefore, up to permutation, $\{\alpha'_1, \alpha'_2, \alpha'_3\} = \{\alpha_1, \alpha_2,\alpha_3\}$ and $\bar{\varphi} \colon  T(\alpha_1, \alpha_2, \alpha_3) \twoheadrightarrow  T(\alpha'_1, \alpha'_2, \alpha'_3)$ is an isomorphism by the hopfian property of triangle groups.
The same argument as in Cases~(V) and (VII) shows that $\Sigma_2(K)$ and $M$ have isomorphic fundamental groups.
\end{proof}

%%%%%%%%%%%%%
\section{Applications to symmetric unions of knots}
\label{sec:symmetric_uinon}

In this section, we use $\pi$-orbifold groups and the previous results to study a classical construction in knot theory introduced by Kinoshita and Terasaka in \cite{KiTe57} and generalized by Lamm~\cite{Lam00}.
This construction is a generalization of the connected sum of a knot with its mirror image, which gives plenty of examples of ribbon knots.

%%%
\subsection{Symmetric unions of knots}

\begin{definition}
\label{def:symmetric_union}
Let $D$ be an unoriented planar diagram of a knot $K_D$ and let $D^*$ be the diagram obtained from $D$ by reflecting $D$ across an axis $\Delta$ in the plane.
Let $B_0, B_1,\dots, B_k$ be balls along the axis $\Delta$, each of which is invariant by the reflection and intersects $D$ in a trivial arc.
One replaces the trivial tangle $(B_0, B_0 \cap (D \cup D^*))$ by an $\infty$-tangle to get the connected sum of the diagrams $D$ and $D^*$.
For $1\leq i \leq k$, one replaces each trivial tangle $(B_i, B_i \cap (D\sharp D^*))$ by a $n_i$-tangle, where $n_i \in \Z$.
The knot diagram $(D \cup D^*)(\infty, n_1, \dots, n_k)$ obtained from $D\cup D^*$ in this way is called a \emph{symmetric union} of the diagram $D$ and $D^*$.
A knot which admits a diagram $(D \cup D^*)(\infty, n_1, \dots, n_k)$ is said to admit a \emph{symmetric union presentation} with \emph{partial knot} $K_D$, where $K_D$ corresponds to the closure of the diagram $D$ such that $(D \cup D^*)(0, 0, \dots, 0) = K_D \cup K_{D}^*$.
See Figure~\ref{fig:symmetric_union}.
\end{definition}

\begin{figure}[h]
 \centering \includegraphics[width=\textwidth]{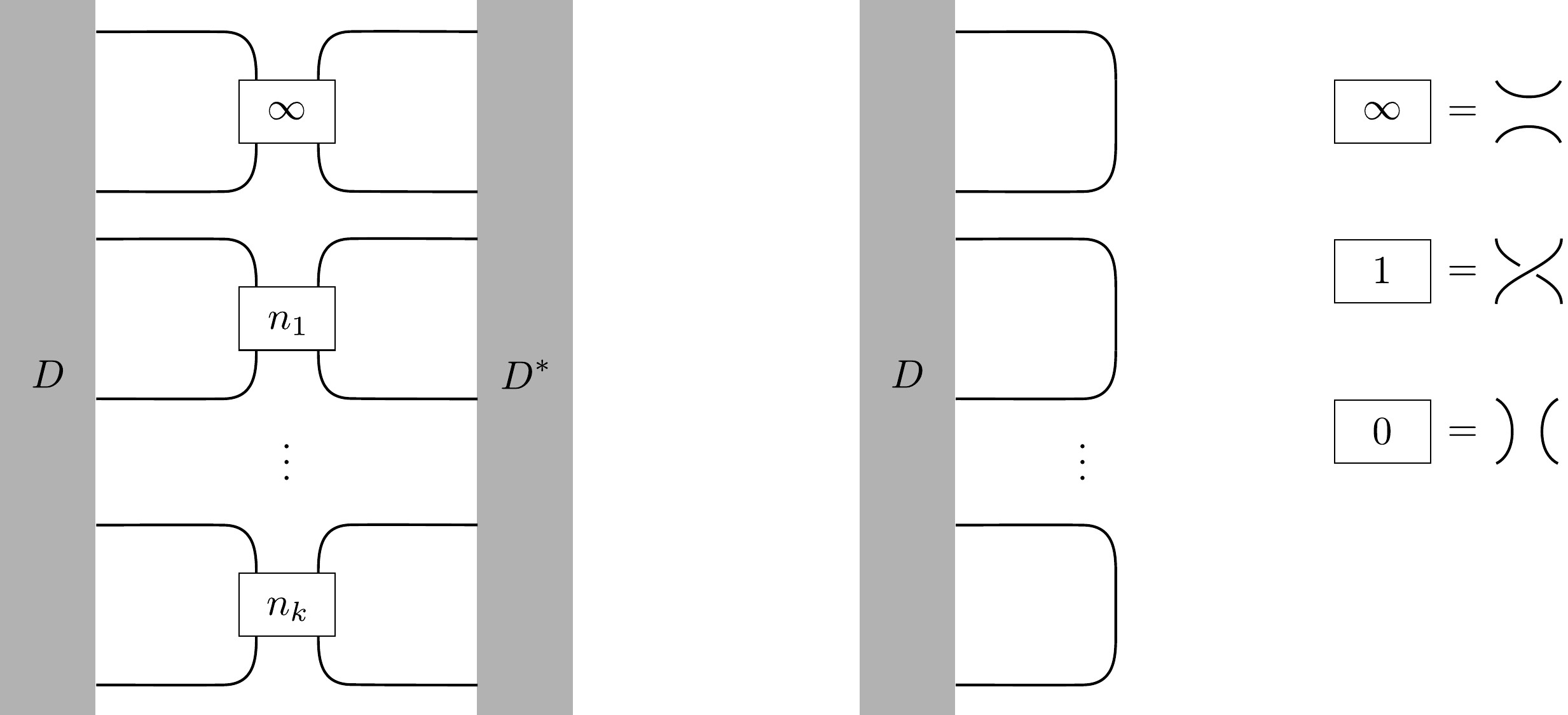}
 \caption{Symmetric union $(D \cup D^*)(\infty, n_1, \dots, n_k)$ and its partial knot $K_D$.}
 \label{fig:symmetric_union}
\end{figure}

The symmetric union construction is not unique and the isotopy type of $K = (D \cup D^*)(\infty, n_1, \dots, n_k)$ depends on both the diagram $D$ and the location of the tangle replacements.
When there is a single tangle replacement, the construction is due to Kinoshita and Terasaka~\cite{KiTe57}.
The extension to multiple symmetric tangle replacements is due to Lamm~\cite{Lam00}.
When the partial knot $K_D$  is oriented and all the twist numbers $n_i$ ($1\leq i \leq k$) are even, the symmetric union $D \cup (-D^\ast)(\infty, n_1, \dots, n_k)$ inherits an orientation from the connected sum $D \sharp (-D^*)$, but when some twists $n_i$ are odd, the orientation of $D \cup D^* (\infty, n_1, \dots, n_k)$ is not well-defined.
A symmetric union $D \cup D^* (\infty, n_1, \dots, n_k)$ is said to be \emph{even} if all the $n_i$ are even.
Otherwise, the symmetric union is said to be \emph{skew}.

The symmetric union construction always produces ribbon knots, and it is still an open question whether there is a ribbon knot that does not admit a symmetric union presentation.
One difficulty in finding such an example is that there is no obstruction known for a ribbon knot to be a symmetric union. 
However, there are candidates like the Montesinos knot $11a_{201} = K(\frac{1}{3}, \frac{2}{3}, \frac{4}{5})$. 
In \cite[Proposition~1.10]{BKN25}, it has been shown that if the knot $11a_{201}$ admits a symmetric union presentation, it has to be skew and the only possible partial knots are $6_1$ and $9_1$. 
At the moment, it has not been possible to rule out these two possibilities of partial knots.
The goal of this section is to pursue the study of the symmetric union construction by investigating more closely the relationship between the knot types of a symmetric union and of its partial knots. 

The following is the key result which relates the $\pi$-orbifold group of a knot and the symmetric union construction. 
It is due to M.~Eisermann (see \cite[Theorem~3.3]{Lam00}).

\begin{proposition}\label{prop:orbifold}
Let $K$ be a symmetric union with partial knot $K_D$.
There is an epimorphism $\varphi_D^\orb\colon G^\orb(K) \twoheadrightarrow G^\orb(K_D)$ which sends the image of a meridian of $K$ to the image of a meridian of $K_D$ and which kills the image of the preferred longitude of $K$.

If $K$ is an even symmetric union, then the epimorphism $\varphi_D$ lifts to an epimorphism $\phi_D\colon G(K) \twoheadrightarrow G(K_D)$ which sends a meridian of $K$ to a meridian of $K_D$ and which kills the preferred longitude of $K$.
\end{proposition}

Using Lemma~\ref{lem:diagram}, Proposition~\ref{prop:orbifold} can be restated as: if $K_D$ is a partial knot for a symmetric union presentation of $K$, then $K \succeq K_D$. 
Then the following result is a corollary of Theorem~\ref{thm:order}. 
It puts some strong restriction on the partial knot types of the symmetric union presentations for some particular families of knots.

\begin{corollary}
\label{cor:partial_knot}
Let $K$ be a knot having a symmetric union presentation with partial knot $K_D$.
Then the following hold.
\begin{enumerate}[label=\textup{(\arabic*)}]
\item If $K$ is trivial, $K_D$ is trivial.
\item If $K$ is a $2$-bridge knot, $K_D$ is a $2$-bridge knot and the symmetric union is skew.
\item If $K$ is a Montesinos knots with $r$ rational tangles \textup{(}$r \geq 3$\textup{)}, $K_D$ is the unknot, or a $2$-bridge knot, or a Montesinos knot with $r'$ rational tangles \textup{(}$r'\leq r$\textup{)}, or a connected sum of $n_1$ $2$-bridge knots and $n_2$ elliptic Montesinos knots with $n_1 + 2n_2 \leq r-1$.
\end{enumerate}
\end{corollary}

When $K$ is a $2$-bridge knot, $\det K = (\det K_D)^2 \neq 1$, then $K_D$ is knotted and hence a $2$-bridge knot.
The symmetric union is skew because an epimorphism of a $2$-bridge knot group onto the group of a non-trivial knot never kills the longitude by 
\cite[Propositions~1.8 and 1.10]{BBRW10}.

Theorem~\ref{thm:arb} also puts some restriction on the partial knots of the symmetric union presentations for arborescent knots.

\begin{corollary}\label{cor:arbpartial}
Let $K$ be an arborescent knot having a symmetric union presentation with partial knot $K_D$. 
Then no factor of a connected sum decomposition of $K_D$ can be a $\pi$-hyperbolic knot.
\end{corollary}

In the same way, the next result is a direct consequence of Theorem~\ref{thm:small}. 
It shows, in particular, that only finitely many distinct partial knots can occur in a symmetric union presentation of a small knot $K$.

\begin{corollary}
\label{cor:finiteness}
Let $K$ be a knot having a symmetric union presentation with partial knot $K_D$. 
If $K$ is a small knot, then the following hold.
\begin{enumerate}[label=\textup{(\arabic*)}]
\item $K_D$ is a small knot.
\item $K_D$ belongs to only finitely many distinct knot types.
\end{enumerate}
\end{corollary}

For a $2$-bridge symmetric union, the finiteness of the possible partial knots has been recently proved by Lamm and Tanaka~\cite{LaTa25} by using the fact that the partial knots are alternating because they have at most two bridges.
See also \cite{Tan24}.

%%%%%
\subsection{Proof of Theorem~\ref{thm:notsymmetricunion}}
We now prove Theorem~\ref{thm:notsymmetricunion} which provides obstructions for a $\pi$-minimal knot $K$ to be a symmetric union, in terms of the rank of the homology of its $2$-fold branched cover $\Sigma_2(K)$.
If $K$ is $\pi$-minimal and admits a symmetric union presentation with partial knot $K_D$, then the $\pi$-orbifold group $G^\orb(K_D)$ is isomorphic to either $G^\orb(K)$, a dihedral group, or $\Z/2\Z$.

The next proposition shows that $G^\orb(K_D)$ cannot be isomorphic to $G^\orb(K)$. 
It gives a positive answer to \cite[Question~(2)]{LaTa25}.

\begin{proposition}\label{prop:notequivalent}
Let $K$ be a symmetric union with partial knot $K_D$. 
If $K$ is non-trivial, then $G^\orb(K)$ is not isomorphic to $G^\orb(K_D)$. 
In particular, $K_D$ cannot be equivalent to $K$.
\end{proposition}

\begin{proof}
We assume that $G^\orb(K)$ is isomorphic to $G^\orb(K_D)$ and look for a contradiction. 
Since the group $G^\orb(K)$ is hopfian, the epimorphism $\varphi_D^\orb\colon G^\orb(K) \twoheadrightarrow G^\orb(K_D)$ given in Proposition~\ref{prop:orbifold} must be an isomorphism. 
The isomorphism $\varphi_D^\orb$ induces an isomorphism $\tilde{\varphi}_D\colon \pi_1(\Sigma_2(K)) \twoheadrightarrow \pi_1(\Sigma_2(K_D))$. 
Therefore,
\[
\det K = \vert H_1(\Sigma_2(K); \Z) \vert = \vert H_1(\Sigma_2(K_D); \Z) \vert = \det K_D.
\]
By \cite[Theorem~2.6]{Lam00}, $\det K= (\det K_D)^2$ and thus $\det K= \det K_D =1$. 
Since $\varphi_D^\orb$ kills the image $\bar{\lambda}_{K}$ of the preferred longitude of $K$, $\bar{\lambda}_{K}$ must be trivial in $G^\orb(K)$. 
Then, the contradiction comes from the following lemma.
\end{proof}

\begin{lemma}\label{lem:longitude}
Let $K$ be a non-trivial knot with $\det K = 1$.
Then the image $\bar{\lambda}_{K}$ of the preferred longitude of $K$ is non-trivial in $G^\orb(K)$.
\end{lemma}

\begin{proof}
Let $\lambda_K \in G(K)$ be a peripheral element representing the preferred longitude of $K$.
Since $\lambda_K$ belongs to the commutator subgroup of $G(K)$, its image $\bar{\lambda}_{K}$ in $G^\orb(K)$ belongs to the subgroup $\pi_1(\Sigma_2(K))$ of index $2$.
Let $p\colon \Sigma_2(K) \to S^3$ denote the covering projection.
Then $\widetilde{K} = p^{-1}(K)$ is the fixed point set of the covering involution $\tau$ on $\Sigma_2(K)$.
In $\pi_1(\Sigma_2(K))$, the element $\bar{\lambda}_{K}$ represents a preferred longitude of the null-homologous knot $\widetilde{K}$.
Since the preferred longitude is homotopic to the knot $\widetilde{K}$, $\bar{\lambda}_{K}$ is trivial in $G^\orb(K)$ if and only if the knot $\widetilde{K}$ is homotopically trivial in $\Sigma_2(K)$.
We use the following claim when the knot $K$ is prime.

\begin{claim} \label{claim:finiteorder}
Let $K$ be a prime knot.
Then, $\widetilde{K}$ has finite order in $\pi_1(\Sigma_2(K))$ if and only if $\pi_1(\Sigma_2(K))$ is finite, that is to say $K$ is a $2$-bridge knot or an elliptic Montesinos knot.
If in addition $\det K = 1$, then $K$ is the $(3, 5)$-torus knot and $\widetilde{K}$ is homotopically non-trivial in $\Sigma_2(K)$.
\end{claim}

If $K$ is a prime knot, it follows from Claim~\ref{claim:finiteorder} that the image $\bar{\lambda}_{K}$ of the preferred longitude of $K$ is non-trivial in $\pi_1(\Sigma_2(K)) \subset G^\orb(K)$.
It remains to consider the case where $K$ is a connected sum of prime knots $K_1,\dots,K_n$.
Then $K = K_1 \sharp K'$ with $K' = K_2 \sharp \cdots \sharp K_{n}$.
Moreover, $\det K_1 = \det K' = 1$ since $\det K = 1$.
The $2$-fold branched cover of $K$ is $\Sigma_2(K) = \Sigma_2(K_1) \sharp \Sigma_2(K')$ and
\[
\bar{\lambda}_{K} = \bar{\lambda}_{K_1}\cdot\bar{\lambda}_{K'} \in \pi_1(\Sigma_2(K)) = \pi_1(\Sigma_2(K_1)) \ast \pi_1(\Sigma_2(K')).
\]
Now, $\bar{\lambda}_{K_1} \neq 1$ in the factor $\pi_1(\Sigma_2(K_1))$ by Claim~\ref{claim:finiteorder}, and $\bar{\lambda}_{K'}$ belongs to the second factor $\pi_1(\Sigma_2(K'))$.
It follows that $\bar{\lambda}_{K} \neq 1$ in $\pi_1(\Sigma_2(K))$.
This completes the proof of Lemma~\ref{lem:longitude}.
\end{proof}

\begin{proof}[Proof of Claim~\ref{claim:finiteorder}]
Let $p\colon \Sigma_2(K) \to S^3$ be the covering projection.
One needs only to prove that if $\pi_1(\Sigma_2(K))$ is infinite, then $\widetilde{K}= p^{-1}(K)$ has infinite order in $\pi_1(\Sigma_2(K))$.
Since $K$ is a prime knot, $\Sigma_2(K)$ is irreducible by the equivariant sphere theorem.
Hence, if $\pi_1(\Sigma_2(K))$ is infinite, then $\Sigma_2(K)$ is aspherical and its universal cover $X $ is contractible (in fact, it is $\R^3$ by the orbifold theorem, see \cite{BoPo01}).
One can lift the covering involution $\tau$ on $\Sigma_2(K)$ to an involution $\tau'$ on the universal cover $X$. Since $X$ is contractible, Smith theory implies that $\Fix(\tau')$ is connected and not compact: it is a component of the preimage of $\widetilde{K}$ in $X$.
Thus $\widetilde{K}$ is of infinite order in $\pi_1(\Sigma_2(K))$. 

By the orbifold theorem, the case where $\pi_1(\Sigma_2(K))$ is finite occurs when $K$ is the unknot, a $2$-bridge knot or an elliptic Montesinos knot with three bridges.
If moreover $\det K = 1$, the $2$-fold branched cover $\Sigma_2(K)$ is $S^3$ or the Poincar\'{e} sphere and the knot $K$ is the unknot or the $(3, 5)$-torus knot which is the unique knot with the Poincar\'{e} sphere as $2$-fold branched cover (see \cite[Affirmation~2.5]{BoOt91}).
In particular, the fixed point set $\widetilde{K}$ of the covering involution is isotopic to a singular fiber of the Seifert fibration of the Poincar\'e sphere $\Sigma_2(K)$.
Thus, $\widetilde{K}$ is homotopically non-trivial in $\Sigma_2(K)$.
\end{proof}

A knot whose $\pi$-orbifold group is dihedral or $\Z/2\Z$ is a $2$-bridge knot or the unknot.
Hence, the next corollary follows from Proposition~\ref{prop:notequivalent} and the definition of $\pi$-minimality.

\begin{corollary}\label{cor:minimalsymmetric} 
If a symmetric union $K$ with partial knot $K_D$ is $\pi$-minimal, then $K_D$ has at most two bridges.
\end{corollary}

Therefore, to prove Theorem~\ref{thm:notsymmetricunion}, it remains to show that the condition on the rank of $H_1(\Sigma_2(K); \Z)$ rules out the possibility of a partial knot $K_D$ with $\leq 2$ bridges. 
This follows from the next proposition since $H_1(\Sigma_2(K); \Z)$ is cyclic or trivial for a knot with $\leq 2$ bridges.

Recall that a finite abelian group $A$ is isomorphic to a direct sum of finite cyclic groups: $G \cong \Z/d_{1}\Z \oplus \cdots \oplus \Z/d_{r}\Z$, where $r =\rank(A)$ and $d_i$ $(\geq 2)$ divides $d_{i+1}$ for $1 \leq i \leq r-1$. 
The invariant factors $\{d_1, \ldots, d_r\}$ are unique.
Let $p$ be a prime number dividing $d_1$ (hence, $p\mid d_i$ for $2 \leq i \leq r$) and let $\F_p$ denote the finite field of order $p$. 
Then, $\dim_{\F_p} (A\otimes \F_p) = r$.

If $K$ admits a symmetric union presentation with partial knot $K_D$, then it is proved in \cite[Theorem~1.3]{Zup26} that 
\[
\col_p(K_D) \leq \col_p(K) \leq \frac{\col_p(K_D)^2}{p},
\]
where $p$ is an odd prime and $\col_p(K)$ denotes the number of $p$-colorings for any diagram of $K$. 
Since $\col_p(K) = p^{1+\dim_{p} H_1(\Sigma_2(K); \F_p)}$, this result implies some inequalities between the mod $p$ first Betti numbers of a symmetric union and of its partial knot. 
This is the content of the next proposition for which we give a direct proof (which is also valid for $p=2$).

\begin{proposition}\label{prop:rank}
If $K$ admits a symmetric union presentation with partial knot $K_D$, then
\[
\rank H_1(\Sigma_2(K_D); \Z) \leq \rank H_1(\Sigma_2(K); \Z) \leq 2\rank H_1(\Sigma_2(K_D); \Z).
\]
\end{proposition} 

\begin{proof} 
By Lemma~\ref{lem:epicover}, the epimorphism $\varphi_D^\orb \colon G^\orb(K) \twoheadrightarrow G^\orb(K_D)$ induces an epimorphism $\tilde{\varphi}_D\colon \pi_1(\Sigma_2(K)) \twoheadrightarrow \pi_1(\Sigma_2(K_D))$. 
Therefore, one gets the first inequality in the statement.

The proof of the second inequality uses the fact that the Goeritz matrix of a diagram of a knot $K$ is a presentation matrix for $H_1(\Sigma_2(K); \Z)$. 
In \cite[Section~3.2]{Lam00}, it is proved that the Goeritz matrix $G_K$ for a symmetric union diagram of $K$ with partial knot $K_D$ is equivalent to
$
\begin{pmatrix}
 G_D & \ast \\ O & -G_D
\end{pmatrix}
$
by a combination of elementary row and column operations, where $G_D$ and $-G_D$ are the Goeritz matrices for the knot diagram $D$ and $D^{*}$, respectively. 
Then, for any field $\F$, one gets $\rank_\F G_K \geq 2\rank_\F G_D$, and thus 
\[
\dim_{\F} H_1(\Sigma_2(K); \F) \leq 2\dim_{\F} H_1(\Sigma_2(K_D); \F).
\]
Letting $\F=\F_p$ for a suitable prime $p$, we obtain the second inequality $\rank H_1(\Sigma_2(K); \Z) \leq 2\rank H_1(\Sigma_2(K_D); \Z)$.
\end{proof}

\begin{corollary}\label{cor:cyclichomology}
Let $K$ be a knot such that $\rank H_1(\Sigma_2(K); \Z) \geq 3$.
If $K$ admits a symmetric union presentation with partial knot $K_D$, then $H_1(\Sigma_2(K_D); \Z)$ cannot be cyclic.
\end{corollary}

We call a knot $K \subset S^3$ \emph{almost $\pi$-minimal} if $K \succeq K'$ implies that $G^\orb(K')$ is either finite or isomorphic to $G^\orb(K)$. 
If $G^\orb(K')$ is finite, then $H_1(\Sigma_2(K'); \Z)$ is cyclic, see \cite[Theorem~2(iii) and (v) in Section~6.2]{Orl72}. 
Therefore, Proposition~\ref{prop:notequivalent} and Corollary~\ref{cor:cyclichomology} imply the following improvement of Theorem~\ref{thm:notsymmetricunion}.

\begin{corollary} \label{cor:notsymmetric} 
Let $K$ be an almost $\pi$-minimal knot. 
If $\rank H_1(\Sigma_2(K); \Z) \geq 3$, then $K$ cannot admit a symmetric union presentation.
\end{corollary}

For example, the proof of \cite[Proposition~5.1]{BKN25} shows that the ribbon knot $11a_{201}$ is almost $\pi$-minimal, but $H_1(\Sigma_2(11a_{201}); \Z)\cong \Z/81\Z$ is cyclic.

It is still open whether the ribbon Montesinos knot $11a_{103} = K(\frac{2}{3}, \frac{1}{3}, \frac{2}{7})$ admits a symmetric union presentation or not. 
The next proposition shows that this knot is $\pi$-minimal. 
Hence, by Corollary~\ref{cor:minimalsymmetric}, the only possible partial knots for a symmetric union presentation of the knot $11a_{103}$ are two $2$-bridge knots $6_1$ or $9_1$ because $\det(11a_{103}) = 81$.

\begin{proposition} \label{prop:11a_{103}} 
The ribbon Montesinos knot $11a_{103} = K(\frac{2}{3}, \frac{1}{3}, \frac{2}{7})$ is $\pi$-minimal.
\end{proposition}

\begin{proof} 
Let us assume that the knot $11a_{103}$ $\pi$-dominates a knot $K$. 
Since $11a_{103}$ is a Montesinos knot with three rational tangles, by Corollary~\ref{cor:knot_order}, $K$ is either a knot with $\leq 2$ bridges or a Montesinos knot with three rational tangles.

Suppose that $K$ is a Montesinos knot $K(\frac{\beta_1}{\alpha_1},  \frac{\beta_2}{\alpha_2}, \frac{\beta_3}{\alpha_3})$. 
Then, by Lemma~\ref{lem:epicover}, the epimorphism $\varphi \colon G^\orb(11a_{103}) \twoheadrightarrow G^\orb(K)$ induces an epimorphism
$\tilde{\varphi}\colon \pi_1(\Sigma_2(11a_{103})) \twoheadrightarrow \pi_1(\Sigma_2(K))$. 
Here, the $2$-fold cover $\Sigma_2(K) = V(0; e_0; \frac{\beta_1}{\alpha_1},  \frac{\beta_2}{\alpha_2}, 
\frac{\beta_3}{\alpha_3})$ is a Seifert fibered manifold with base $S^2(\alpha_1, \alpha_2, \alpha_3)$ and rational Euler number 
$e_0 =  \sum_{i=1}^{3} \frac{\beta_i}{\alpha_i}  = \frac{|H_1(\Sigma_2(K); \Z)|}{\alpha_1 \alpha_2 \alpha_3}$. 
In the same way, $\Sigma_2(11a_{103}) = V(0; \frac{9}{7}; \frac{2}{3}, \frac{1}{3}, \frac{2}{7})$ is Seifert fibered with base $S^2(3, 3, 7)$ and rational Euler number $\frac{9}{7} = \frac{|H_1(\Sigma_2(K); \Z)|}{63}$, and hence $|H_1(\Sigma_2(K); \Z)| = 81$.

The epimorphism $\tilde{\varphi}$ induces an epimorphism $\bar{\varphi}\colon T(3, 3, 7) \twoheadrightarrow  T(\alpha_1, \alpha_2, \alpha_3)$ through taking the quotients by the centers. 
The presentation of the triangle group $T(3, 3, 7) =  \langle x, y, z \mid x^{3} = y^{3} = z^{7} =xyz = 1 \rangle$ shows that each image $\bar{\varphi}(x)$, $\bar{\varphi}(y)$, and $\bar{\varphi}(z)$ is non-trivial, otherwise the image $\bar{\varphi}(T(3, 3, 7))$ would be either trivial or $\Z/3\Z$. 
Hence, $\bar{\varphi}(T(3, 3, 7)) = T(\alpha_1, \alpha_2, \alpha_3)$ must contain elements of orders $3$ and $7$. 
We distinguish two cases according to whether $\frac{1}{\alpha_1} + \frac{1}{\alpha_2} + \frac{1}{\alpha_3} \geq 1$ or not.

If $\frac{1}{\alpha_1} + \frac{1}{\alpha_2} + \frac{1}{\alpha_3} \geq 1$, then the only possibilities for the triple $\{\alpha_1, \alpha_2, \alpha_3\}$ are $\{2, 3, 3\}$, $\{2, 3, 5\}$, and $\{3, 3, 3\}$ since at most one of $\alpha_i$ is even because $\Sigma_2(K)$ is a $\Z/2\Z$-homology sphere. 
This case is impossible since neither $T(2, 3, 3)$, $T(2, 3, 5)$, nor $T(3, 3, 3)$ contains an element of order $7$.

When $\frac{1}{\alpha_1} + \frac{1}{\alpha_2} + \frac{1}{\alpha_3} < 1$, $\bar{\varphi}(T(3, 3, 7)) = T(\alpha_1, \alpha_2, \alpha_3) \subset \PSL(2,\mathbb{R})$ is a hyperbolic triangle group generated by two elliptic elements  $a = \bar{\varphi}(y)$ and $b = \bar{\varphi}(z)$ of orders $3$ and $7$, respectively.
Then, up to taking suitable powers $u=a^{k}$ and $v = b^{\ell}$ to normalize the matrix representatives of $u$ and $v$ in $\SL(2,\R)$, at least one of Cases~(I)--(VII) in \cite[Theorem~2.3]{Kna68} holds. 
Since the triangle group $ T(\alpha_1, \alpha_2, \alpha_3)$ is co-compact, Case~(II) is impossible by \cite[Figures~2 and 3]{Kna68}.
Cases~(III) and (VI) do not hold because the generators $u$ and $v$ do not have the same order. 
Case~(IV) is not possible because neither of the generators $u$ nor $v$ has order $2$. 
Cases (V) and (VII) correspond to the case where $u$ and $v$ generate a hyperbolic triangle group of type $T(2, 3, 7)$ by \cite[Section~3]{Kna68}. 
Hence, among Cases~(II)--(VII), the only possibility left is that $\bar{\varphi}(T(3, 3, 7)) = T(2, 3, 7)$, which is ruled out by Lemma~\ref{lem:noepimorphism} below.

In Case~(I), $uv$ has extreme negative trace and by \cite[Proposition~2.2]{Kna68}, $u$ and $v$ generate a hyperbolic triangle group 
$T(3, 7, n)$. 
The epimorphism $\bar{\varphi} \colon  T(3, 3, 7)$ $\twoheadrightarrow  T(3, 7, n),$
implies that $-\chi(S^2(3, 3, 7)) \geq -\chi(S^2(3, 7, n))$ by \cite[Lemma~2.5]{Ron92}. 
So, $1 - (\frac{1}{3} + \frac{1}{3} + \frac{1}{7}) \geq 1 - (\frac{1}{3} + \frac{1}{7} + \frac{1}{n})$, and hence $n \leq 3$.
Assume $n= 3$. 
Since triangle groups are hopfian, the epimorphism $\bar{\varphi}$ is an isomorphism.
Furthermore, the restriction of $\tilde{\varphi} \colon \pi_1(\Sigma_2(11a_{103})) \twoheadrightarrow \pi_1(\Sigma_2(K))$ to the center $Z(\pi_1(\Sigma_2(11a_{103})))\cong \Z$ is injective since $\pi_1(\Sigma_2(K))$ is torsion-free.
Therefore, the epimorphism $\tilde{\varphi}$ is an isomorphism and, by the five lemma, so is $\varphi \colon G^\orb(11a_{103}) \twoheadrightarrow G^\orb(K)$.
Therefore, like in Cases~(II)--(VII) above, the only possibility is that $\bar{\varphi}(T(3, 3, 7)) = T(2, 3, 7)$, which is ruled out by Lemma~\ref{lem:noepimorphism} below.
This completes the proof.
\end{proof}

\begin{lemma} \label{lem:noepimorphism}
There is no epimorphism $\tilde{\varphi}\colon \pi_1(V(0; \frac{9}{7}; \frac{2}{3}, \frac{1}{3}, \frac{2}{7})) \twoheadrightarrow
\pi_1(V(0; e_0; \frac{\beta_1}{2},  \frac{\beta_2}{3}, \frac{\beta_3}{7}))$.
\end{lemma}

\begin{proof}
Suppose, to the contrary, that there is an epimorphism $\tilde{\varphi}\colon \pi_1(M) \twoheadrightarrow \pi_1(N)$. 
Since $\beta_2$ is coprime to $3$, $\vert H_1(N; \Z) \vert = 21 \beta_1 + 14 \beta_2 + 6 \beta_3$ is coprime to $3$. 
Moreover, $\vert H_1(N; \Z) \vert$ must divide $\vert H_1(M; \Z) \vert = 81$, and hence  $\vert H_1(N; \Z) \vert = 1$. 
So $N$ is the Brieskorn homology sphere $ V(0; \frac{1}{42}; \frac{1}{2}, \frac{-1}{3}, \frac{-1}{7})$, see \cite[Theorem~7.12]{JaNe83}. 
Since $M$ and $N$ are aspherical, the epimorphism $\tilde{\varphi}$ can be realized by a map $f \colon M \to N$. 
Since both $\pi_1(M)$ and $\pi_1(N)$ have rank $2$, the $\pi_1$-surjective map $f$ is of non-zero degree by \cite[Theorem~2.1]{RWZ02}. 
Then, by \cite[Theorem~3.2 and Remark~2]{Ron93}, $f$ is homotopic to $g\circ \pi$, where $\pi$ is a composition of finitely many vertical pinches and $g$ is a fiber-preserving branched covering. 
Since $M$ has only three exceptional fibers, a vertical pinch on  $M$ would produce a Seifert fibered manifold with $\leq 2$ exceptional fibers, and hence with cyclic fundamental group. 
Therefore, up to homotopy, $f$ cannot factorize through a vertical pinch since $\pi_1(N)$ is not cyclic.
It follows that $f$ is homotopic to a fiber-preserving branched covering $g$. 
Let $d_1$ be the degree of $g$ along the fiber and $d_2$ the degree of the induced map $\bar{g} \colon S^2(3, 3, 7) \to S^2(2, 3, 7)$.
Then, $\frac{4}{21} = \vert \chi( S^2(3, 3, 7))\vert \geq d_2 \vert \chi( S^2(2, 3, 7)) \vert = \frac{d_2}{42}$. 
This implies that $8\geq d_2$. 
Moreover, $\frac{9}{7} = e_0(M) = \frac{d_2}{d_1} e_0(N) = \frac{d_2}{42 d_1}$ by \cite[Theorem~1.2]{NeRa78}. 
Then $d_1 = \frac{d_2}{54} \leq \frac{8}{54}$, which is impossible.
\end{proof}

Proposition~\ref{prop:rank} allows us to give some lower bounds for the bridge number, genus, and unknotting number of the partial knot of a symmetric union presentation.
For a knot $K$, let $b(K)$ denote the \emph{bridge number}, $g(K)$ the \emph{Seifert genus}, and $u(K)$ the \emph{unknotting number} of $K$.

\begin{corollary}\label{cor:bridgenumber}
If $K$ admits a symmetric union presentation with partial knot $K_D$, then
\begin{enumerate}[label=\textup{(\arabic*)}]
\item $b(K_D) \geq \frac{1}{2} \rank H_1(\Sigma_2(K); \Z) + 1$.
\item $g(K_D) \geq \frac{1}{4} \rank H_1(\Sigma_2(K); \Z)$.
\item $u(K_D) \geq \frac{1}{2} \rank H_1(\Sigma_2(K); \Z)$.
\end{enumerate}
\end{corollary}

\begin{proof} 
The proofs follow from Proposition~\ref{prop:rank} and the following well-known facts.
First, $b(K_D)$ satisfies the inequality $b(K_D) \geq \rank \pi_1(\Sigma_2(K_D)) + 1$, see \cite[Lemma~1.6]{BoZi85}.
Next, for a Seifert matrix $S$ of $K_D$, the sum $S+S^T$ gives a presentation matrix for the $\Z$-module $H_1(\Sigma_2(K_D); \Z)$, and therefore $2g(K_D) \geq \rank H_1(\Sigma_2(K_D); \Z)$.
Finally, Wendt's inequality~\cite{Wen37} shows that $u(K_D) \geq \rank H_1(\Sigma_2(K_D); \Z)$ (see also \cite{Kin57}).
\end{proof}

\begin{remark}
In some cases, Wendt's inequality is powerful.
For $K=13a_{1786}$, we have $u(K)\leq 3$ since a single crossing change produces $10_{75}$ whose unknotting number is $2$.
According to KnotInfo, the smooth $4$-genus of $K$ equals $1$, and hence one cannot use it to determine $u(K)$.
On the other hand, the invariant factors of $H_1(\Sigma_2(K); \Z)$ are $\{3, 3, 33\}$, and thus Wendt's inequality implies $u(K)=3$.
\end{remark}

%%%%%%%%%%%
\section{The bridge number and preorder $\succeq$}\label{sec:bridge}

Corollary~\ref{cor:bridge_number} raises a question of a relationship between the $\pi$-domination relation $\succeq$ and the bridge number $b(L)$ of a link $L$.

\begin{question}\label{ques:bridge_number} 
Let $L, L' \subset S^3$ be prime links, does $L \succeq L'$ imply $b(L) \geq b(L')$?
\end{question}

The answer to Question~\ref{ques:bridge_number} is positive when $L$ is a Montesinos knot by Corollary~\ref{cor:bridge_number}.
It is also true when $L'$ has at most three bridges, and hence when $G^\orb(L')$ is finite (by the orbifold theorem). 

\begin{lemma} \label{lem:bridge_finite} 
Let $L, L' \subset S^3$ such that $L \succeq L'$. 
If $b(L') \leq 3$, then $b(L) \geq b(L')$.
\end{lemma}

\begin{proof}
The inequality is clear when $b(L') = 1$.
If $b(L') = 2$, then Theorem~\ref{thm:order}(1) implies that $L$ cannot be the unknot, and thus $b(L) \geq 2$. 
If $b(L') = 3$, then Theorem~\ref{thm:order}(1) and (2) imply that $b(L) \geq 3$.
\end{proof}

Question~\ref{ques:bridge_number} is also positive for torus knots.

\begin{proposition}\label{prop:bridge_torusknot} 
If $K_1$ is a torus knot and $K_1 \succeq K_2$, then $b(K_1) \geq b(K_2)$.
\end{proposition}

\begin{proof} 
Since a torus knot is a small knot, $K_2$ is a small knot by Theorem~\ref{thm:small}(1). 
Then $K_2$ is the unknot, a $2$-bridge knot, an elliptic Montesinos knot, or a torus knot with infinite orbifold group by Theorem~\ref{thm:order}(4). 
If $K_2$ is the unknot, a $2$-bridge knot, or an elliptic Montesinos knot, then 
$b(K_2) \leq 3$ and $b(K_1) \geq b(K_2)$ by Lemma~\ref{lem:bridge_finite}.
Hence we are left to consider the case where $K_1$ and $K_2$ are both torus knots with infinite orbifold groups. 

Let us assume that $K_1$ is a $(p_1, q_1)$-torus knot and $K_2$ a $(p_2 ,q_2)$-torus knot with $3 \leq p_1 < q_1$, $3 \leq p_2 < q_2$, $p_1, q_1$ coprime, and $p_2, q_2$ coprime. 
Furthermore, if $p_i = 3$, then $q_i \geq 7$. 
Since $K_1 \succeq K_2$ there is an epimorphism $\varphi\colon G^\orb(K_1) \twoheadrightarrow G^\orb(K_2)$ which, by the proof of Theorem~\ref{thm:order}(4), sends the  infinite cyclic center $Z_1$ of $G^\orb(K_1)$ into the infinite cyclic center $Z_2$ of $G^\orb(K_2)$. 
The quotients $G(K_1) /Z_1 \cong \pi_{1}^\orb (S^2(2, p_1, q_1)) =T(2, p_1, q_1)$ and $G(K_2) /Z_2 \cong \pi_{1}^\orb (S^2(2, p_2, q_2)) =T(2, p_2, q_2)$ are hyperbolic triangle groups since $\frac{1}{2} + \frac{1}{p_i} + \frac{1}{q_i} < 1$ for $i= 1, 2$. 
In particular, they are centerless, and thus the epimorphism $\varphi$ induces an epimorphism $\bar{\varphi} \colon T(2, p_1, q_1) \twoheadrightarrow T(2, p_2, q_2)$.
Like in the proof of Proposition~\ref{prop:minimal_example}, it follows from the presentation of the triangle group $T(2, p_1, q_1)$ that each image $\bar{\varphi}(x)$, $\bar{\varphi}(y)$ and $\bar{\varphi}(z)$ is not trivial, otherwise the image $\bar{\varphi}(T(2, p_1, q_1))$ would be trivial.
Therefore, the non-trivial elements $a = \bar{\varphi}(x)$ and $b = \bar{\varphi}(y)$ satisfies $a^{2} = 1$ and $b^{p'_1} = 1$ with $p'_1$ dividing $p_1$. 
The elements $a$ and $b$ generate the image $\bar{\phi}( T(2, p_1, q_1)) = T(2, p_2, q_2)$ which is a discrete subgroup of $\PSL(2, \R)$ 
since $T(2, p_2, q_2)$ is a hyperbolic triangle group. 
As in the proof of Proposition~\ref{prop:minimal_example}, up to taking suitable powers $u=a^{\pm 1}$ and $v = b^{k}$, $k$ coprime to $p'_1$ to normalize the matrix representatives of $u$ and $v$ in $SL(2,\R)$, it follows that at least one of Cases (I)--(VII) in \cite[Theorem~2.3]{Kna68} holds. 
Since the triangle group $T(2, p_2, q_2)$ is co-compact, Case~(II) is impossible by \cite[Figures~2 and 3]{Kna68}.
Case~(III) cannot happens, otherwise $0 = \tr(u) = \tr(v)$ and $\bar{\varphi}(T(2, p_1, q_1)) = T(2, \ell, n)$ with $\tr(u) = \tr(v) = \cos(\pi - \frac{\pi}{\ell})$ by \cite[Section~3]{Kna68}. 
Then $\ell = 2$, which contradicts $\frac{1}{2} + \frac{1}{\ell} + \frac{1}{n} < 1$ because $T(2, \ell, n) = T(2, p_2, q_2)$ is a hyperbolic triangle group. 
If Case~(IV) happens, then $u$ and $v$ generate a hyperbolic triangle group $T(2, 3, n)$ by \cite[Section~3]{Kna68}. 
Since $T(2, 3, n) = T(2, p_2, q_2)$ with $3 \leq p_2 < q_2$ and $p_2$ coprime to $q_2$, it follows that $p_2= 3$ and $q_2 = n$. 
Therefore,
\[
b(K_1) = \min\{p_1, q_1\} = p_1 \geq 3 = p_2 = \min\{p_2, q_2\} = b(K_2).
\]
Cases~(V)--(VII) do not hold because $u$ is of order $2$.
If Case~(I) occurs, then $u$ and $v$ generate a triangle group $T(2, p'_1, n)$ where $n$ is the order of $uv$. Since $T(2, p'_1, n) = T(2, p_2, q_2)$, then either $p'_1 = p_2$ or $p'_1 = q_2$. 
In any case, $p_2 \leq p'_1 \leq p_1$ since $p_2 < q_2$. 
Therefore, as above, $b(K_1) = \min\{p_1, q_1\} = p_1 \geq p_2 = \min\{p_2, q_2\} = b(K_2)$.
\end{proof}

For a prime link with infinite orbifold group, a positive answer to Question~\ref{ques:bridge_number} would follow from a positive answer to the following question which is a stronger version of the well-known meridional rank conjecture by Cappell and Shaneson~\cite[Problem~1.11]{Kir97} about whether there is an equality between the bridge number $b(L)$ and the meridional rank $m(L)$ of a link $L$.

\begin{question}\label{ques:m2rc}
Let $L$ be a prime link with infinite orbifold group and let $m_2(L)$ denote the minimal number of torsion elements needed to generate $G^\orb(L)$. 
Is it true that $b(L) = m_{2}(L)$?
\end{question}

\begin{remark}\label{rem:order_2}
For a link $L$ as in Question~\ref{ques:m2rc}, $\Sigma_2(L)$ is aspherical by Lemma~\ref{lem:folklore}(3). 
Therefore, a torsion element in $G^\orb (L)$ has order $2$ and is the image of a meridian of $L$.
\end{remark}

\begin{remark}\label{rem:m2} 
If $L, L' \subset S^3$ are two prime links with infinite orbifold groups and $L \succeq L'$, then $m_2(L) \geq m_2(L')$ by definition.
\end{remark}

\begin{lemma}\label{lem:bridge_torsion}
Let $L, L' \subset S^3$ be two prime links with infinite orbifold groups. 
If $L \succeq L'$ and $b(L') = m_2(L')$ then $b(L) \geq b(L')$.
\end{lemma}

\begin{proof} 
It follows from Remark~\ref{rem:order_2} that $b(L) \geq m(L) \geq m_2(L)$ for such a link. 
In particular, if $L \succeq L'$ and $b(L') = m_2(L')$, then $b(L) \geq m_2(L) \geq m_2(L') =b(L')$.
\end{proof}

It follows from results in \cite{BBK21} and \cite{BBKM23} that the answer to Question~\ref{ques:m2rc} is positive for some special classes of links including Montesinos links, twisted links (see \cite{BBK21} for the definition), and arborescent links associated to bipartite trees with even weights. 
In these articles, it is shown that for such a link, a lower bound on the meridional rank obtained via Coxeter quotients of the link group matches the bridge number of the link. 
A \emph{Coxeter quotient} of a link group $G(L)$ is a Coxeter group $C$ with an epimorphism from $G(L)$ onto $C$ which sends the meridians of $G(L)$ to reflections in $C$. 
The \emph{reflection rank} of a Coxeter group $C$ is the minimal number of reflections needed to generate $C$. 
The \emph{Coxeter rank} $\cox(L)$ of a link $L$ is the maximum of the reflection ranks of the Coxeter quotients of $G(L)$. 
For a prime link with infinite orbifold group, the inequalities 
\[
b(L) \geq m(L) \geq m_2(L) \geq \cox(L)
\]
follow from the proof of Lemma~\ref{lem:bridge_torsion} and \cite[Proposition~1]{BBK21}.
In particular, $b(L) = \cox(L)$ implies that $b(L) = m(L) = m_2(L)$.

%%%%%%%%%%%%%%%%%%%%%%%%%%

\end{document}